\documentclass[11pt]{article}
\usepackage{amsmath,amssymb,amsthm}
\usepackage{amscd}
\newtheorem{theorem}{Theorem}[section]
\newtheorem{lemma}[theorem]{Lemma}
\newtheorem{proposition}[theorem]{Proposition}
\newtheorem{corollary}[theorem]{Corollary}

\theoremstyle{plain}

\theoremstyle{definition}

\newtheorem{remark}[theorem]{Remark}

\numberwithin{equation}{section}

\renewcommand{\labelenumi}{\textup{(\theenumi)}}

\newcommand{\Homeo}{\operatorname{Homeo}}

\newcommand{\id}{\operatorname{id}}
\newcommand{\Ker}{\operatorname{Ker}}
\newcommand{\sgn}{\operatorname{sgn}}
\newcommand{\Ad}{\operatorname{Ad}}

\def\det{{{\operatorname{det}}}}

\newcommand{\N}{\mathbb{N}}

\newcommand{\R}{\mathbb{R}}
\newcommand{\RZ}{\mathbb{R}/{\mathbb{Z}}}
\newcommand{\Z}{\mathbb{Z}}
\newcommand{\Zp}{{\mathbb{Z}}_+}
\newcommand{\f}{\tilde{f}}

\title{Continuous orbit equivalence,
flow equivalence  of Markov shifts 
and circle actions on Cuntz--Krieger algebras}
\author{Kengo Matsumoto \\
Department of Mathematics \\
Joetsu University of Education \\
Joetsu, 943-8512, Japan
}
\date{}

\begin{document}
\maketitle

\def\det{{{\operatorname{det}}}}

\begin{abstract}
We will study circle actions 
on Cuntz--Krieger algebras 
trivially acting on its canonical maximal 
abelian $C^*$-subalgebras from the view points of 
continuous orbit equivalence of one-sided topological Markov shifts
and  
flow equivalence of two-sided topological Markov shifts.
\end{abstract}




\def\OA{{{\mathcal{O}}_A}}
\def\OB{{{\mathcal{O}}_B}}
\def\OZ{{{\mathcal{O}}_Z}}
\def\OTA{{{\mathcal{O}}_{\tilde{A}}}}
\def\SOA{{{\mathcal{O}}_A}\otimes{\mathcal{K}}}
\def\SOB{{{\mathcal{O}}_B}\otimes{\mathcal{K}}}
\def\SOZ{{{\mathcal{O}}_Z}\otimes{\mathcal{K}}}
\def\SOTA{{{\mathcal{O}}_{\tilde{A}}\otimes{\mathcal{K}}}}
\def\DA{{{\mathcal{D}}_A}}
\def\DB{{{\mathcal{D}}_B}}
\def\DZ{{{\mathcal{D}}_Z}}
\def\DTA{{{\mathcal{D}}_{\tilde{A}}}}
\def\SDA{{{\mathcal{D}}_A}\otimes{\mathcal{C}}}
\def\SDB{{{\mathcal{D}}_B}\otimes{\mathcal{C}}}
\def\SDZ{{{\mathcal{D}}_Z}\otimes{\mathcal{C}}}
\def\SDTA{{{\mathcal{D}}_{\tilde{A}}\otimes{\mathcal{C}}}}
\def\Max{{{\operatorname{Max}}}}
\def\Per{{{\operatorname{Per}}}}
\def\PerB{{{\operatorname{PerB}}}}
\def\Homeo{{{\operatorname{Homeo}}}}
\def\HA{{{\frak H}_A}}
\def\HB{{{\frak H}_B}}
\def\HSA{{H_{\sigma_A}(X_A)}}
\def\Out{{{\operatorname{Out}}}}
\def\Aut{{{\operatorname{Aut}}}}
\def\Ad{{{\operatorname{Ad}}}}
\def\Inn{{{\operatorname{Inn}}}}
\def\det{{{\operatorname{det}}}}
\def\exp{{{\operatorname{exp}}}}
\def\cobdy{{{\operatorname{cobdy}}}}
\def\Ker{{{\operatorname{Ker}}}}
\def\ind{{{\operatorname{ind}}}}
\def\id{{{\operatorname{id}}}}
\def\supp{{{\operatorname{supp}}}}
\def\co{{{\operatorname{co}}}}
\def\Sco{{{\operatorname{Sco}}}}
\def\ActA{{{\operatorname{Act}_{\DA}(\mathbb{R}/{\mathbb{Z}},\OA)}}}
\def\ActB{{{\operatorname{Act}_{\DB}(\mathbb{R}/{\mathbb{Z}},\OB)}}}
\def\RepOA{{{\operatorname{Rep}(\mathbb{R}/{\mathbb{Z}},\OA)}}}
\def\RepDA{{{\operatorname{Rep}(\mathbb{R}/{\mathbb{Z}},\DA)}}}
\def\RepDB{{{\operatorname{Rep}(\mathbb{R}/{\mathbb{Z}},\DB)}}}
\def\U{{{\mathcal{U}}}}

\section{Introduction and Preliminaries}

For an $N\times N$ irreducible matrix
$A=[A(i,j)]_{i,j=1}^N$ 
with entries in $\{0,1\}$,
let us denote by 
$X_A$ the shift space 
\begin{equation}
X_A = \{ (x_n )_{n \in \N} \in \{1,\dots,N \}^{\N}
\mid
A(x_n,x_{n+1}) =1 \text{ for all } n \in {\N}
\} \label{eq:Markovshift}
\end{equation}
of the one-sided topological Markov shift for $A$.
We assume that 
$1< N \in {\N}$
and
$A$ is not any permutation matrix.
Hence 
$A$  satisfies condition (I) in the sense of Cuntz--Krieger \cite{CK}.
We endow $X_A$ with the relative topology 
of the product topology
on $\{1,\dots,N\}^{\N}$, so that 
$X_A$ is a compact Hausdorff space  
and homeomorphic to the Cantor discontinuum.
Let us denote by
$\sigma_A$ 
the shift map on $X_A$ which is defined by 
$\sigma_{A}((x_n)_{n \in {\N}})=(x_{n+1} )_{n \in \N}$.
It is a continuous surjection on $X_A$.
The topological dynamical system 
$(X_A, \sigma_A)$ is called the one-sided topological Markov shift for $A$.
The topological Markov shifts are often called 
shifts of finite type (cf. \cite{Kitchens}, \cite{LM}).
The two-sided topological Markov shift
is similarly defined as a topological 
dynamical system
$(\bar{X}_A, \bar{\sigma}_A)$
on the set $\bar{X}_A$
of two-sided
sequences 
$(x_n)_{n \in \Z}$ 
instead of one-sided sequences 
$(x_n )_{n \in \N}$ in \eqref{eq:Markovshift}
with the shift homeomorphism
$\bar{\sigma}_{A}((x_n)_{n \in {\Z}})=(x_{n+1} )_{n \in \Z}$
on 
$\bar{X}_A$.
J. Cuntz and W. Krieger have introduced a $C^*$-algebra 
$\OA$ associated to topological Markov shift $(X_A,\sigma_A)$
(\cite{CK}).
The $C^*$-algebra is called the Cuntz--Krieger algebra,
which is a universal unique $C^*$-algebra generated by
partial isometries
$S_1,\dots,S_N$
subject to the relations:
\begin{equation} 
\sum_{j=1}^N S_j S_j^* = 1, \qquad
S_i^* S_i = \sum_{j=1}^N A(i,j) S_jS_j^*, \quad i=1,\dots,N. \label{eq:CK}
\end{equation} 
For $t \in \RZ$,
the correspondence
$S_i \rightarrow e^{2 \pi\sqrt{-1}t}S_i,
\, i=1,\dots,N$
gives rise to an automorphism
of $\OA$ denoted by
$\rho^A_t$.
The automorphisms
$\rho^A_t, t \in \RZ$
yield an action of the circle group $\RZ$
on $\OA$ called the gauge action.
Let us denote by $\DA$ the $C^*$-subalgebra of $\OA$
generated by the projections of the form:
$S_{i_1}\cdots S_{i_n}S_{i_n}^* \cdots S_{i_1}^*,
i_1,\dots,i_n =1,\dots,N$.
The subalgebra $\DA$ is canonically isomorphic to the commutative 
$C^*$-algebra $C(X_A)$
of the complex valued continuous functions on $X_A$
by identifying the projection
$S_{i_1}\cdots S_{i_n}S_{i_n}^* \cdots S_{i_1}^*
$
with the characteristic function
$\chi_{U_{i_1\cdots i_n}} \in C(X_A)$
of the cylinder set
$U_{i_1\cdots i_n}$
for the word
${i_1\cdots i_n}$.
Cuntz and Krieger have proved that if one-sided topological Markov shifts $(X_A,\sigma_A)$ and   
$(X_B,\sigma_B)$ are topologically conjugate,
then there exists an isomorphism 
$\varPhi: \OA\rightarrow \OB$
such that
$\varPhi(\DA) = \DB$
and
$\varPhi\circ \rho^A_t = \rho^B_t \circ \varPhi, t \in \RZ$
(\cite[2.17 Proposition]{CK}).

Continuous orbit equivalence of one-sided topological Markov shifts
is a weaker equivalence relation than one-sided topological conjugacy
 and
 gives rise to isomorphic Cuntz--Krieger algebras (\cite{MaPacific}).
Let $A$ and $B$ be irreducible square matrices with entries in $\{0,1 \}$.
One-sided topological Markov shifts
$(X_A, \sigma_A)$ and $(X_B,\sigma_B)$ 
are said to be continuously orbit equivalent 
if there exist a homeomorphism
$h: X_A \rightarrow X_B$ 
and continuous functions 
$k_1,l_1: X_A\rightarrow \Zp,
 k_2,l_2: X_B\rightarrow \Zp 
$
such that
\begin{align}
\sigma_B^{k_1(x)} (h(\sigma_A(x))) 
& = \sigma_B^{l_1(x)}(h(x))
\quad 
\text{ for} \quad 
x \in X_A,  \label{eq:orbiteq1x} \\
\sigma_A^{k_2(y)} (h^{-1}(\sigma_B(y))) 
& = \sigma_A^{l_2(y)}(h^{-1}(y))
\quad 
\text{ for } \quad 
y \in X_B. \label{eq:orbiteq2y}
\end{align}
In \cite{MMKyoto}, it has been proved that 
the following three assertions are equivalent
(cf. \cite{MaPacific}, \cite{MaPAMS}, 
\cite{MatuiPLMS}, \cite{MatuiPre2012}):
\begin{enumerate}
\renewcommand{\theenumi}{\roman{enumi}}
\renewcommand{\labelenumi}{\textup{(\theenumi)}}
\item
 $(X_A, \sigma_A)$ and $(X_B,\sigma_B)$ 
are continuously orbit equivalent.
\item There exists an isomorphism $\Phi:\mathcal{O}_A\to\mathcal{O}_B$ 
such that $\Phi(\mathcal{D}_A)=\mathcal{D}_B$. 
\item $\mathcal{O}_A$ and $\mathcal{O}_B$ 
are isomorphic and $\sgn(\det(\id-A))= \sgn(\det(\id-B))$. 
\end{enumerate}

Let $g$ be a strictly positive continuous function on $\bar{X}_A$.
Denote by 
$\bar{X}_A^g$
the compact Hausdorff space obtained from
\begin{equation*}
\{ (\bar{x}, r ) \in \bar{X}_A\times\R \mid \bar{x}\in\bar{X}_A,
0 \le r \le g(x) \} 
\end{equation*}
by identifying
$(\bar{x}, g(\bar{x}))$
and
$(\bar{\sigma}_A(\bar{x}),0)$
for each $\bar{x} \in \bar{X}_A$.
Let
$
\bar{\sigma}_{A,t}^g, t \in \R
$
be the flow on
$\bar{X}_A^g$
defined by 
\begin{equation*}
\bar{\sigma}_{A,t}^g([(\bar{x},r)]) 
=[(\bar{x},r+t)]
\quad
\text{ for }
[(\bar{x},r)] \in \bar{X}_A^g.
\end{equation*}
The dynamical system
$(\bar{X}_A^g, \bar{\sigma}_A^g)$
is called the suspension flow of 
$(\bar{X}_A, \bar{\sigma}_A)$
by ceiling function $g$.
Two-sided topological Markov shifts
$(\bar{X}_A, \bar{\sigma}_A)$
and
$(\bar{X}_B, \bar{\sigma}_B)$
are said to be flow equivalent if 
for a strictly positive continuous function $g$ on $\bar{X}_A$
there exists 
a strictly positive continuous function $\tilde{g}$ on $\bar{X}_B$
such that 
$(\bar{X}_A^g, \bar{\sigma}_A^g)$
is topologically conjugate to
$(\bar{X}_B^{\tilde{g}}, \bar{\sigma}_B^{\tilde{g}})$.
Throughout the paper,
denote by
${\mathcal{K}}$
the $C^*$-algebra of compact operators on 
the separable infinite dimensional Hilbert space
$\ell^2(\N)$
and
${\mathcal{C}}$
its maximal abelian  $C^*$-subalgebra consisting of diagonal elements on
$\ell^2(\N)$.
We have seen that 
the following three assertions are equivalent
(\cite{MMKyoto}, \cite{Franks}, \cite{Ro}
cf. \cite{BF}, \cite{CK}, \cite{PS}) :
\begin{enumerate}
\renewcommand{\theenumi}{\roman{enumi}}
\renewcommand{\labelenumi}{\textup{(\theenumi)}}
\item
 $(\bar{X}_A, \bar{\sigma}_A)$ and 
 $(\bar{X}_B, \bar{\sigma}_B)$ 
are flow equivalent.
\item There exists an isomorphism 
$\bar{\Phi}:\SOA\to\SOB$ 
such that $\bar{\Phi}(\SDA)=\SDB$. 
\item $\SOA$ and $\SOB$ 
are isomorphic and $\sgn(\det(\id-A))= \sgn(\det(\id-B))$. 
\end{enumerate}
Hence the continuous orbit equivalence of one-sided topological Markov shifts
is regarded as a one-sided counterpart of 
the flow equivalence of two-sided topological Markov shifts.

Let us denote by
$(\bar{H}^A,\bar{H}^A_+)$
the ordered cohomology group 
for the two-sided topological Markov shift $(\bar{X}_A,\bar{\sigma}_A)$,
which is defined as the quotient group of the ordered abelian group
$C(\bar{X}_A,{\mathbb{Z}})$ of all ${\mathbb{Z}}$-valued continuous functions 
on $\bar{X}_A$ by the subgroup
$\{ \xi - \xi \circ \bar{\sigma}_A \mid \xi \in C(\bar{X}_A,{\mathbb{Z}}) \}$.
The classes of nonnegative functions in $C(\bar{X}_A,{\mathbb{Z}})$ 
form the positive cone $\bar{H}^A_+$ 
(cf. \cite{BH}, \cite{Po}).
The ordered cohomology group
$(H^A,H^A_+)$
for the one-sided topological Markov shift $(X_A,\sigma_A)$
is similarly defined (\cite{MMKyoto}).
The latter ordered group  
$(H^A,H^A_+)$ is naturally isomorphic to the former one  
$(\bar{H}^A,\bar{H}^A_+)$ (\cite[Lemma 3.1]{MMKyoto}).
In \cite{BH}, Boyle--Handelman have proved that 
the ordered group
$(\bar{H}^A,\bar{H}^A_+)$ 
is a complete invariant for  flow equivalence  of 
two-sided topological Markov shift
$(\bar{X}_A, \bar{\sigma}_A)$.
Suppose that $(X_A, \sigma_A)$ and $(X_B,\sigma_B)$ 
are continuously orbit equivalent
via a homeomorphism $h:X_A\rightarrow X_B$
satisfying \eqref{eq:orbiteq1x} and \eqref{eq:orbiteq2y}.
The homomorphism
$\Psi_h: C(X_B,\Z) \rightarrow C(X_A,\Z)$
defined by
\begin{equation}
\Psi_{h}(f)(x)
= \sum_{i=0}^{l_1(x)-1} f(\sigma_B^i(h(x))) 
- \sum_{j=0}^{k_1(x)-1} f(\sigma_B^j(h(\sigma_A(x))))
\label{eq:Psihfx}
\end{equation}
for
$f \in C(X_B,\Z), \ x \in X_A$
and its inverse 
$\Psi_{h^{-1}}: C(X_A,\Z) \rightarrow C(X_B,\Z)$
give rise to isomorphisms of abelian groups
between $C(X_A,\Z)$ and $C(X_B,\Z)$.
They furthermore 
induce isomorphisms between  
$(H^A,H^A_+)$ and 
$(H^B,H^B_+)$
as ordered groups
(\cite{MMPre2014}).

In this paper, we will study actions of  the circle group 
$\RZ$ on the Cuntz--Krieger algebras 
from the view points of 
continuous orbit equivalence of one-sided topological Markov shifts 
and
flow equivalence of two-sided topological Markov shifts.
Let us denote by $\ActA$ 
the set of actions of 
${\RZ}$ on $\OA$  
 trivially acting on the maximal abelian $C^*$-subalgebra 
$\DA$.
For $f \in C(X_A, \Z)$,
define the one-parameter unitary group
$U_t(f), t \in \RZ$ in $\DA$
by 
\begin{equation}
U_t(f) = \exp({2\pi \sqrt{-1} t f}), \label{eq:utf}
\end{equation}
and an automorphism
$\rho_t^{A,f}$ on $\OA$ for each $t \in \RZ$ 
by 
\begin{equation}
\rho_t^{A,f}(S_i) = U_t(f) S_i, \qquad i=1,\dots,N. \label{eq:rhotf}
\end{equation}
We know that 
$\rho_t^{A,f}$ belongs to
$\ActA$.
For $\alpha \in \ActA$,
a continuous map
$t \in \RZ \rightarrow u_t \in \OA$
is called a unitary one-cocycle relative to $\alpha$ 
if $u_t, t \in \RZ$ are unitaries
and satisfy 
$u_{t+s} = u_t \alpha_t(u_s), t,s \in \RZ$. 
For
$\alpha^1, \alpha^2 \in \ActA$,
we write
$\alpha^1\sim \alpha^2$
if there exists a unitary one-cocycle 
$u_t \in \OA$ relative to 
$\alpha^2$ 
such that
$\alpha^1_t = \Ad u_t \circ \alpha^2_t$
for all $t \in \RZ$.
In Section 2, 
we show that the set $\ActA$ has a natural structure of ordered group.
We will then see that
the map
$f \in C(X_A,\Z) \rightarrow \rho^{A,f} \in \ActA$
gives rise to an isomorphism of groups and induces
an isomorphism
from $H^A$ onto $\ActA/\sim$ as ordered groups
(Proposition \ref{prop:HACT}).
The action $\rho^{A,1}$ for the constant function $1$
is the gauge action $\rho^A$
which is an order unit of $\ActA$.
In Section 3,
we will see that 
the map
$\Psi_h:C(X_B,\Z) \rightarrow C(X_A,\Z)$
defined in 
\eqref{eq:Psihfx}
controls actions of $\RZ$ on the Cuntz--Krieger algebras
trivially acting  on its maximal abelian $C^*$-subalgebras
as in the following theorem.
\begin{theorem}[{Theorem \ref{thm:COETFAE}}]
Suppose that the one-sided topological Markov shifts
$(X_A, \sigma_A)$ and $(X_B,\sigma_B)$
are continuously orbit equivalent via a homeomorphism 
$h:X_A\rightarrow X_B$.
Then there exists an isomorphism
$\Phi:\OA \rightarrow \OB$ satisfying the following conditions
\begin{enumerate}
\renewcommand{\theenumi}{\roman{enumi}}
\renewcommand{\labelenumi}{\textup{(\theenumi)}}
\item
$\Phi(\DA) = \DB$,
\item
$
\Phi \circ \rho^{A,f}_t = \rho^{B,\Psi_{h^{-1}}(f)}_t \circ \Phi
\text{ for }
f \in C(X_A,\Z), \,\, 
 t \in \RZ,
$
\item 
$\Psi_{h^{-1}}:C(X_A,\Z) \rightarrow C(X_B,\Z)$ 
induces 
an isomorphism from  
$H^A$
to  $H^B$
as ordered abelian groups.
\end{enumerate}
\end{theorem}
If one may take 
$l_1(x) = k_1(x) +1, x \in X_A$
in \eqref{eq:orbiteq1x}
and
$l_2(y) = k_2(y) +1, y \in X_B$
in \eqref{eq:orbiteq2y},
then 
$(X_A,\sigma_A)$ and $(X_B,\sigma_B)$
are said to be {\it eventually one-sided conjugate\/}.
By using above theorem, we then obtain the following result.
\begin{theorem}[{Corollary \ref{coro:evstcoe}}]
There exists an isomorphism
$\Phi:\OA \rightarrow \OB$ 
such that 
\begin{equation*}
\Phi(\DA) = \DB
\quad 
\text{  and }
\quad
\Phi \circ \rho^A_t = \rho^B_t \circ \Phi, 
\qquad t \in {\RZ}
\end{equation*}
if and only if 
$(X_A, \sigma_A)$ and $(X_B,\sigma_B)$
are eventually one-sided conjugate.
\end{theorem}
For flow equivalence of two-sided topological Markov shifts,
we have the following theorem
by using Parry--Sullivan's theorem \cite{PS} 
which says that the flow equivalence relation is 
generated by topological conjugacies of two-sided topological Markov shifts and expansions of the underlying directed graphs.
\begin{theorem}[{Theorem \ref{thm:FETFAE}}]
Suppose that the two-sided topological Markov shifts
$(\bar{X}_A, \bar{\sigma}_A)$ and $(\bar{X}_B,\bar{\sigma}_B)$
are flow equivalent.
Then there exist an isomorphism
$\Phi:\SOA \rightarrow \SOB$,
a homomorphism
$\varphi:C(X_A,\Z) \rightarrow C(X_B,\Z)$
of ordered abelian groups
and 
a unitary one-cocycle
 $u_t^f \in \U(M(\SOA))$ relative to 
$\rho^{A,f} \otimes\id$
for each function
$f\in C(X_A,\Z)$
satisfying the following conditions
\begin{enumerate}
\renewcommand{\theenumi}{\roman{enumi}}
\renewcommand{\labelenumi}{\textup{(\theenumi)}}
\item
$\Phi(\SDA) = \SDB$,
\item
$\Phi \circ \Ad(u_t^f) \circ (\rho^{A,f}_t\otimes \id)
 = (\rho^{B,\varphi(f)}_t\otimes\id) \circ \Phi
\text{ for }
 f \in C(X_A,\Z), \, t \in \RZ,
$
\item
$\varphi:C(X_A,\Z) \rightarrow C(X_B,\Z)$
induces 
an isomorphism
from  
$H^A$
to  $H^B$
as ordered abelian groups.
\end{enumerate}
\end{theorem}
\medskip

In this paper,
we denote 
by $\N$
the set of positive integers
and
by $\Zp$
the set of nonnegative integers,
respectively.
For a topological Markov shift
$(X_A,\sigma_A)$,
a word $\mu =(\mu_1, \dots, \mu_k)$ 
for $\mu_i \in \{1,\dots,N\}$
is said to be admissible for $X_A$ 
if there exists 
$(x_n)_{n \in \N} \in X_A$ 
such that 
$(\mu_1, \dots, \mu_k) =(x_1, \dots, x_k)$.
The length of $\mu$ is $k$, 
which is denoted by $|\mu|$.
 We denote by 
$B_k(X_A)$ the set of all admissible words of length $k$.
The cylinder set 
$\{ (x_n )_{n \in \Bbb N} \in X_A 
\mid x_1 =\mu_1,\dots, x_k = \mu_k \}$
for $\mu=(\mu_1,\dots,\mu_k) \in B_k(X_A)$
is denoted by $U_\mu$.  

\medskip

This paper is a revised version of the paper entitled 
{\it Continuous orbit equivalence, flow equivalence of Markov shifts and torus actions on Cuntz--Krieger algebras}, arXiv:1501.06965.

\section{Circle actions}
We fix  
an $N\times N$ irreducible matrix
$A=[A(i,j)]_{i,j=1}^N$ with entries in $\{0,1\}$.
Let $S_1,\dots, S_N$
be the canonical generating partial isometries of the Cuntz-Krieger 
algebra $\OA$ satisfying the relations \eqref{eq:CK}.
For a word $\mu =(\mu_1, \dots, \mu_k) \in B_k(X_A)$,
denote by $S_\mu$ the partial isometry
$S_{\mu_1}\cdots S_{\mu_k}$.  
Let us denote by
${\operatorname{Act}}(\RZ,\OA)$
the set of all continuous actions of  the circle group
${\mathbb{R}}/{\mathbb{Z}} $ as automorphisms of $\OA$.
We set
\begin{equation}
\operatorname{Act}_{\DA}(\RZ,\OA)
=\{ \alpha \in \operatorname{Act}(\RZ,\OA) \mid
\alpha_t(a) = a \text{ for all } t \in \RZ, a \in \DA\}.
\end{equation}
We denote by 
$\operatorname{Rep}(\RZ,\OA)$
(resp.
$\operatorname{Rep}(\RZ,\DA)$)
the set of unitary representations of $\RZ$ into 
the unitary group of $\OA$ (resp. of $\DA$).
We are assuming that the matrix $A$ 
is irreducible and not any permutation matrix,
so that the subalgebra $\DA$ is maximal abelian in $\OA$ (\cite{CK}).
Hence if $\alpha \in \ActA$ is of the form 
$\alpha_t = \Ad(u_t), t \in \RZ$
for some  
$u \in \operatorname{Rep}(\RZ,\OA)$,
then
$u_t \in \DA$ for all $t \in \RZ$
and hence 
$u \in \operatorname{Rep}(\RZ,\DA)$.

Recall that 
for a function $f\in C(X_A,\Z)$ 
and $t \in \RZ$,
an automorphism
$\rho_t^{A,f}\in\Aut(\OA)$ 
is defined by 
$
\rho_t^{A,f}(S_i) = U_t(f) S_i, i=1,\dots,N, t \in \RZ
$
for the unitary
$
U_t(f) = \exp({2\pi \sqrt{-1} t f})\in \DA
$
as in \eqref{eq:utf} and \eqref{eq:rhotf}.
It is easy to see that
the automorphisms
$\rho_t^{A,f}, t \in \RZ$ 
yield an action of 
$\RZ$ on $\OA$
 such that
$\rho^{A,f}_t(a) =a$
for all $a \in \DA$
 so that
$\rho^{A,f} \in \ActA$.
\begin{lemma}
An action $\alpha \in \operatorname{Act}(\RZ,\OA)$
belongs to
 $\ActA$ if and only if 
there exists a continuous function
$f \in C(X_A,\Z)$ 
such that 
$\alpha_t = \rho_t^{A,f}$ 
for all
$t \in \RZ$.
\end{lemma}
\begin{proof}
It suffices to show the only if part.
Suppose that $\alpha \in \operatorname{Act}(\RZ,\OA)$
belongs to $\ActA$.
By \cite[Lemma 4.6, Corollary 4.7]{MaJOT2000},
there exists a unitary $u_t \in {\U}(\DA)$
for each $t \in \RZ$
 such that
$\alpha_t(S_i) = u_t S_i$ for 
$i=1,\dots,N$ (cf. \cite{Cu2}).
As $\alpha_s(u_t)= u_t$
for $s,t \in \RZ$,
we see that for $i=1,\dots,N$
\begin{equation*}
u_t u_s S_i 
= u_t \alpha_s(S_i) 
= \alpha_s(u_t S_i) 
= \alpha_s(\alpha_t(S_i)) 
= \alpha_{s+t}(S_i) 
= u_{s+t} S_i 
\end{equation*}
so that
$u_t u_s = u_{s+t}$.
The continuity of $u:\RZ \rightarrow  {\U}(\DA)$
follows from 
that of the action $\alpha$.
Hence
there exists a unitary representation
$u \in \RepDA$
satisfying
$\alpha_t(S_i) = u_t S_i$ for 
$i=1,\dots,N$ and $t \in \RZ$.
As in \cite[Lemma 6.4]{MaPre2014},
there exists a continuous function
$f \in C(X_A,\Z)$ 
such that 
$u_t = U_t(f)$ 
and hence
$\alpha_t(S_i) = \rho_t^{A,f}(S_i)$ 
for  $i=1,\dots,N, \,t \in \RZ$.
\end{proof}
For $\alpha, \beta \in \ActA$,
take 
$f,g \in C(X_A,\Z)$
such that
$
\alpha_t(S_i) = U_t(f)S_i,
$
$ 
\beta_s(S_i) = U_s(g)S_i,  
s,t \in \RZ, i=1,\dots,N.
$
The commutativity  
$
U_s(g)U_t(f)= U_t(f)U_s(g)
$
implies 
$\alpha_t \circ \beta_s = \beta_s \circ \alpha_t$
for $s,t \in \RZ$. 
Hence the set $\ActA$ has a structure of abelian group by the product
$(\alpha\cdot\beta)_t = \alpha_t\circ\beta_t$ for $t \in \RZ$.
As the set $C(X_A,\Z)$ 
has a structure of abelian group by pointwise sums of functions,
we get the following.
\begin{proposition} \label{prop:abelgr}
The correspondence
$
f \in C(X_A,\Z) \rightarrow \rho^{A,f} \in \ActA
$
yields an isomorphism of abelian groups.
\end{proposition}
\begin{proof}
Since the identities
\begin{equation*}
U_t(f_1 + f_2) = U_t(f_1) U_t(f_2),
\qquad
U_t(-f) = U_{t}(f)^*,
\qquad
U_t(0) = 1
\end{equation*}
hold for
$f_1, f_2, f \in C(X_A,\Z)$,
we have
\begin{equation*}
\rho_t^{A,f_1 + f_2} = \rho_t^{A,f_1}\circ \rho_t^{A,f_2},
\qquad
\rho_t^{A,-f} = (\rho_{t}^{A,f})^{-1},
\qquad
\rho_t^{A,0} = \id.
\end{equation*}
Hence the correspondence
$f \in C(X_A,\Z) \rightarrow \rho^{A,f} \in \ActA$
yields a homomorphism of groups.
Its surjectivity comes from the preceding lemma.
It is clear to see that $\rho_t^{A,f} = \id$ for all $t \in \RZ$
if and only if $f$ is identically zero.
\end{proof}
Under the identification
between
the algebras 
$C(X_A)$ and $\DA$
through the correspondence 
between
$\chi_{U_\mu} \in C(X_A)$
and
$S_\mu S_\mu^* \in \DA$
for $\mu \in B_k(X_A)$,
we have
$ \sum_{i=1}^N S_i f S_i^* = f \circ \sigma_A$
for $f \in C(X_A)$
so that the equalities
\begin{equation}
S_i U_t(f) 
= \sum_{j=1}^N S_j U_t(f) S_j^* S_i 
= U_t(f\circ\sigma_A) S_i, 
\quad i=1,\dots,N \label{eq:siu}
\end{equation}
hold.
For a function $f \in C(X_A,\Z)$, 
denote by $[f]$ the class of $f$ in the ordered cohomology group
$H^A$. 
We then have the following lemma. 
\begin{lemma} \label{lem:HA}
For $f_1, f_2 \in C(X_A,\Z)$,
the following two conditions are equivalent:
\begin{enumerate}
\renewcommand{\theenumi}{\roman{enumi}}
\renewcommand{\labelenumi}{\textup{(\theenumi)}}
\item
$[f_1] = [f_2]$ in $H^A$.
\item
$\rho_t^{A,f_2} = \Ad(u_t) \circ \rho_t^{A,f_1}$ on $\OA$
for some $u \in \RepDA$.
\end{enumerate}
\end{lemma}
\begin{proof}
(i) $\Rightarrow$ (ii):
Since $[f_1] = [f_2]$ in $H^A$,
there exists a continuous function $b \in C(X_A,\Z)$ such that
$f_2 = f_1 + b- b\circ \sigma_A$.
Put
$u_t = U_t(b) \in {\U}(\DA)$ for $t \in \RZ$.
By the identity \eqref{eq:siu}, one sees that
\begin{equation*}
u_t \rho_t^{f_1}(S_i) u_t^*
 = U_t(b) U_t(f_1) S_i U_t(b)^* 
 = U_t(b) U_t(f_1) U_t(-b \circ \sigma_A) S_i 
  = U_t(f_2) S_i 
\end{equation*} 
so that
we have
$\Ad(u_t) \circ \rho_t^{A,f_1} = \rho_t^{A,f_2}$.

(ii) $\Rightarrow$ (i):
By \cite[Lemma 6.4]{MaPre2014},
there exists a continuous function
$b \in C(X_A,\Z)$ 
such that 
$u_t = U_t(b)$. 
The equality
$\rho_t^{A,f_2}(S_i) = \Ad(u_t) \circ \rho_t^{A,f_1}(S_i) $
implies that
\begin{equation*}
U_t(f_2) S_i
= U_t(b + f_1 - b\circ \sigma_A) S_i, \qquad i=1,\dots,N 
\end{equation*}
so that
$f_2 = f_1 + b- b\circ \sigma_A$
and hence
$[f_1] = [f_2]$.
\end{proof}
For
$\alpha^1, \alpha^2 \in \ActA$,
we write
$\alpha^1\sim \alpha^2$
if there exists a unitary one-cocycle 
$u_t  \in \OA, t \in \RZ$
relative to $\alpha^2$
such that
$\alpha^1_t = \Ad u_t \circ \alpha^2_t$
for $t \in \RZ$.
In this case, we know that 
$u_t$ belongs to $\DA$
and satisfies
$u_{t+s} = u_t u_s$ for $t,s \in \RZ$
because $\alpha^i_t(a) =a$ for all $a \in \DA$
and $\DA$ is maximal abelian in $\OA$.
It is easy to see that 
$\sim$ is an equivalence relation in $\ActA$.
We write
$[\alpha]$ for the equivalence class of $\alpha \in \ActA$. 
For $\alpha \in \ActA$, we put
\begin{equation*}
\delta(\alpha)_t = \frac{d}{dt}\left(\sum_{i=1}^N \alpha_t(S_i)S_i^*\right)(t).
\end{equation*}
Define 
$\alpha \geq 0$ if 
$\frac{1}{2\pi\sqrt{-1}}\delta(\alpha)_0 $ is a positive operator in $\OA$. 
We write
$[\alpha]\geq 0$ 
if there exists $\alpha' \in \ActA$ such that
$\alpha'\geq 0$ and $\alpha \sim \alpha'$. 
As
$\delta(\rho^{A,f})_0 = 2\pi\sqrt{-1} f$,
we have
\begin{proposition}\label{prop:HACT}
The isomorphism
$f \in C(X_A,\Z) \rightarrow \rho^{A,f} \in \ActA$
of groups induces an isomorphism
$[f] \in H^A\rightarrow [\rho^{A,f}] \in \ActA/\sim$
 of ordered abelian groups.
\end{proposition}

\section{Continuous orbit equivalence and circle actions}
Let
$h : X_A \rightarrow X_B$
be a homeomorphism
which gives rise to a continuous orbit equivalence
between
$(X_A, \sigma_A)$
and
$(X_B,\sigma_B)$.
Take $k_1, l_1 \in C(X_A,\Zp)$
satisfying \eqref{eq:orbiteq1x}.
For $f \in C(X_B,\Z )$, define
$\Psi_{h}(f) \in C(X_A,\Z)$ 
by the formula
\eqref{eq:Psihfx}.
The map
$\Psi_h: C(X_B,\Z) \rightarrow C(X_A,\Z)$
does not depend on the choice of the functions 
$k_1, l_1$ 
as long as they are satisfying \eqref{eq:orbiteq1x}
(\cite[Lemma 4.2]{MMPre2014}).
Thus $\Psi_h: C(X_B,\Z) \rightarrow C(X_A,\Z)$
gives rise to a homomorphism of abelian groups,
which is actually an isomorphism (\cite[Proposition 4.5]{MMPre2014}).
As the equality
$\Psi_h(f - f \circ\sigma_B) = f \circ h - f \circ h \circ \sigma_A$
holds,
$\Psi_h$ induces a homomorphism
from $H^B$ to $H^A$
which yields an isomorphism of ordered groups
(\cite{MMKyoto}, \cite{MMPre2014}).
For $f \in C(X_A,\Z)$ 
and $n \in \Zp$, let us denote by $f^n$ 
the function 
$f^n(x) =\sum_{i=0}^{n-1}f (\sigma_A^i(x)), x \in X_A$.
The following identity 
is useful in our further discussions.
\begin{lemma}
For a word
$\mu = (\mu_1,\dots,\mu_n) \in B_n(X_A)$,
we have
\begin{equation}
\rho_t^{A,f}(S_\mu) = U_t(f^n)S_\mu,
\qquad f \in C(X_A,\Z),\, t \in \RZ.
\end{equation}
\end{lemma}
\begin{proof}
By the identity 
$S_i U_t(f)
=U_t(f\circ \sigma_A) S_i
$
as in 
\eqref{eq:siu},
we have 
\begin{align*}
\rho_t^{A,f}(S_{\mu_1 \cdots \mu_n}) 
& = U_t(f)S_{\mu_1}\cdots  U_t(f)S_{\mu_{n-2}}
      U_t(f)S_{\mu_{n-1}} U_t(f)S_{\mu_n} \\
& = U_t(f)S_{\mu_1} \cdots U_t(f)S_{\mu_{n-2}} 
      U_t(f + f\circ \sigma_A)S_{\mu_{n-1}} S_{\mu_n} \\
& = U_t(f + f\circ \sigma_A + \cdots +f\circ \sigma_A^{n-1})
S_{\mu_1}\cdots S_{\mu_{n-2}} S_{\mu_{n-1}} S_{\mu_n} \\
& = U_t(f^n)S_\mu.
\end{align*}
\end{proof}

The following result is a generalization of 
\cite[2.17 Proposition]{CK}.
\begin{theorem}\label{thm:COETFAE}
If the one-sided topological Markov shifts
$(X_A, \sigma_A)$ and $(X_B,\sigma_B)$
are continuously orbit equivalent
via a homeomorphism $h:X_A\rightarrow X_B$,
then there exists an isomorphism
$\Phi:\OA \rightarrow \OB$ such that 
\begin{equation*}
\Phi(\DA) = \DB
\quad 
\text{  and }
\quad
\Phi \circ \rho^{A,\Psi_{h}(f)}_t = \rho^{B,f}_t \circ \Phi
\quad
\text{ for }
f \in C(X_B,\Z), \,\, 
 t \in \RZ.
\end{equation*}
\end{theorem}

\begin{proof}
The first part of the proof below follows essentially the proof of 
\cite[Proposition 5.6.]{MaPacific}
(cf. \cite[Proposition 6.3]{MaPre2014}).
We need the several notations used in the first part to proceed to
the second part, so that the first part
overlaps   
\cite[Proposition 5.6.]{MaPacific}
and
\cite[Proposition 6.3]{MaPre2014}.
Let us denote by 
$\HA$ (resp. $\HB$) the Hilbert space
with its complete orthonormal system
$\{ e_x^A \mid x \in X_A\}$ 
(resp. $\{ e_y^B \mid y \in X_B\}$).
Define the partial isometries 
$S_i^A, i=1,\dots,N$ on $\HA$
 by 
\begin{equation}
S_i^A e_x^A =
\begin{cases}
e_{ix}^A & \text{ if } ix \in X_A,\\
0 & \text{ otherwise,}
\end{cases}
\label{eq:siAex}
\end{equation} 
which 
satisfy the relations
\eqref{eq:CK}. 
For the $M \times M$ matrix $B =[B(i,j)]_{i,j=1}^M$,
define the 
partial isometries  $S_i^B, i=1,\dots,M$ 
on $\HB$
satisfying the relations
\eqref{eq:CK} for $B$
in a similar way.
The Cuntz--Krieger algebra 
$\OA$ (resp.  $\OB$) 
is identified with
the $C^*$-algebra 
$C^*(S_1^A,\dots, S_N^A)$  on $\HA$
(resp. $C^*(S_1^B,\dots, S_M^B)$  on $\HB$)
generated by $S_1^A,\dots, S_N^A$
(resp. $S_1^B,\dots, S_M^B$).
For the continuous function 
$
k_1:X_A \rightarrow \Zp
$ 
in \eqref{eq:orbiteq1x},
let
$K_1 = \Max\{k_1(x) \mid x \in X_A\}$.
By adding 
$K_1 - k_1(x)$ to both $k_1(x)$ and $l_1(x)$,
one may assume that
$k_1(x) = K_1$ for all $x \in X_A$.

We will see that
$\Phi = \Ad (U_h)$
defined by 
the unitary
$
U_h e_x^A = e_{h(x)}^B,
x \in X_A
$
satisfies the desired properties.
We fix $i \in \{1,\dots,N\}$
and
set
$X_B^{(i)} = \{ y \in X_B \mid ih^{-1}(y) \in X_A \}$.
For $y \in X_B^{(i)}$,
put
$z = ih^{-1}(y) \in X_A$.
As $h(\sigma_A(z)) =y$
with \eqref{eq:orbiteq1x}, 
one has
$h(z) 
\in  \sigma_B^{-l_1(z)}(\sigma_B^{K_1}(y))$
and 
\begin{equation}
h(z) = (\mu_1(z),\dots,\mu_{l_1(z)}(z), y_{K_1+1}, y_{K_1+2}, \dots ) 
\label{eq:hz}
\end{equation}
for some
$(\mu_1(z),\dots,\mu_{l_1(z)}(z)) \in B_{l_1(z)}(X_B)$.
Put
$
L_1= \Max\{l_1(z) \mid z = i h^{-1}(y), \, y \in X_B^{(i)} \}.
$
The set 
$
W^{(i)} =\{(\mu_1(z),\dots,\mu_{l_1(z)}(z)) \in B_{l_1(z)}(X_B) 
\mid z = i h^{-1}(y), \, y \in X_B^{(i)} \}
$
of words 
is a finite subset of
$\cup_{j=0}^{L_1}B_j(X_B)$.
For $\nu=(\nu_1,\dots, \nu_j) \in W^{(i)}$,
put the clopen set
$
E_\nu^{(i)}
$
 in $X_B^{(i)}$
$$
E_\nu^{(i)} =\{ y \in X_B^{(i)} 
\mid \mu_1(z) = \nu_1,\dots,\mu_{l_1(z)}(z) =\nu_j, \, z = i h^{-1}(y)  \}
$$
so that 
$
X_B^{(i)} 
= \cup_{\nu \in W^{(i)} }E_\nu^{(i)}.
$
Let
$
Q_\nu^{(i)} 
$
be
the characteristic function
$
\chi_{E_\nu^{(i)}}
$
on $X_B$
for the clopen sets 
$
E_\nu^{(i)}.
$
For $y \in X_B^{(i)}$ 
and
$ \nu \in W^{(i)}$,  
we have
$ y \in E_\nu^{(i)}$ 
if and only if
$Q_\nu^{(i)}e_y^B = e_y^B$.
By \eqref{eq:hz},
the equality
\begin{equation*}
e_{h(ih^{-1}(y))}^B 
= \sum_{\nu \in W^{(i)}}
 S_{\nu}^B \sum_{\xi\in B_{K_1}(X_B)}{S_\xi^B}^* Q_\nu^{(i)} e_y^B
 \quad
 \text{ for } y \in X_B^{(i)}
\end{equation*}
holds. 
As
$
U_h S_i^A U_h^* e_y^B 
= 
e_{h(ih^{-1}(y))}^B 
$
if
$y \in X_B^{(i)},$
and
$0$
otherwise, we have
\begin{equation*}
U_h S_i^A U_h^* e_y^B  
= \sum_{\nu \in W^{(i)}}\sum_{\xi\in B_{K_1}(X_B)}
 S_{\nu}^B {S_\xi^B}^* Q_\nu^{(i)} e_y^B
\qquad \text{ for } y \in X_B^{(i)}
\end{equation*}
so that 
\begin{equation*}
U_h S_i^A U_h^*  
= \sum_{\nu \in W^{(i)}} \sum_{\xi\in B_{K_1}(X_B)}
 S_{\nu}^B {S_\xi^B}^* Q_\nu^{(i)}.
\end{equation*}
As $ Q_\nu^{(i)}\in \DB$,
we have
$\Ad(U_h) (S_i^A) \in \OB$  
so that 
$\Ad(U_h) (\OA) \subset \OB.$  
Since $U_h^* = U_{h^{-1}}$, 
we symmetrically have
$\Ad(U_h^*) (\OB) \subset \OA$
so that 
$\Ad(U_h) (\OA) = \OB.$
It is direct to see that 
$\Ad(U_h) (f) = f \circ h^{-1}$ for $ f \in \DA$
from the definition $U_h e_x^A = e_{h(x)}^B, x \in X_A$
so that we have
$\Ad(U_h) (\DA) = \DB.$

We will next show that 
$
\Ad(U_h) \circ  \rho^{A,\Psi_h(f)}_t
= \rho^{B,f}_t \circ \Ad(U_h)
$
for $f \in C(X_B,\Z),
t \in \RZ.
$
It follows that 
\begin{equation}
(\rho^{B,f}_t \circ \Ad(U_h)) (S_i^A) 
 =  \sum_{\nu \in W^{(i)}} \sum_{\xi\in B_{K_1}(X_B)}
    \rho_t^{B,f}(S_{\nu}^B {S_\xi^B}^*) Q_\nu^{(i)}. \label{eq:rhoBtf}
\end{equation}
Since
$Q_\nu^{(i)}e_y^B \ne 0$ if and only if
$Q_\nu^{(i)}e_y^B  = e_y^B$ and $\nu_1 = \mu_1(z),\dots, \nu_j = \mu_{l_1(z)}(z)$.
For $y \in E_\nu^{(i)}$
with $(y_1,\dots,y_{K_1}) = (\xi_1,\dots,\xi_{K_1})$, 
we have 
$S_{\nu}^B {S_\xi^B}^* Q_\nu^{(i)}e_y^B = e_{h(z)}^B = e_{h(ih^{-1}(y))}^B$
and
$\nu \sigma_B^{K_1}(y) = h(z).$
As
\begin{equation*}
\rho_t^{B,f}({S_\xi^B}^*)e_y^B 
=\exp\{-2\pi \sqrt{-1}t f^{K_1}(y) \}{S_\xi^B}^* e_y^B 
=\exp\{-2\pi \sqrt{-1}t f^{K_1}(h(\sigma_A(z)))\}
e_{\sigma_B^{K_1}(y)}^B 
\end{equation*}
and
\begin{equation*}
\rho_t^{B,f}(S_{\nu}^B) e_{\sigma_B^{K_1}(y)}^B
  = U_t(f^{l_1(z)} )  e_{\nu \sigma_B^{K_1}(y)}^B   
  =\exp\{2\pi \sqrt{-1}t f^{l_1(z)} (h(z))\} e_{h(z)}^B, 
\end{equation*}
we have  
\begin{align*}
    \rho_t^{B,f}(S_{\nu}^B {S_\xi^B}^* Q_\nu^{(i)} ) e_y^B 
=  & \rho_t^{B,f}(S_{\nu}^B) \rho_t^{B,f}({S_\xi^B}^*)e_y^B \\
=  & \rho_t^{B,f}(S_{\nu}^B) 
     \exp\{-2\pi \sqrt{-1}t f^{K_1}(h(\sigma_A(z)))\}
     e_{\sigma_B^{K_1}(y)}^B  \\
=  & \exp\{-2\pi \sqrt{-1}t f^{K_1}(h(\sigma_A(z))) \}
     \rho_t^{B,f}(S_{\nu}^B) e_{\sigma_B^{K_1}(y)}^B  \\
=  & \exp\{-2\pi \sqrt{-1}t f^{K_1}(h(\sigma_A(z))) \}
     \exp\{2\pi \sqrt{-1}t f^{l_1(z)}(h(z)) \}
     e_{h(z)}^B \\
= & \exp\{2\pi \sqrt{-1}t(\Psi_h(f)(z))\} e_{h(z)}^B \\ 
= & U_h \exp\{2\pi \sqrt{-1}t\Psi_h(f)\} S_i^A e_{h^{-1}(y)}^A \\
= & U_h \rho_t^{A,\Psi_h(f)}(S_i^A ) U_h^* e_y^B. 
\end{align*}
By \eqref{eq:rhoBtf}, we get
\begin{equation*}
(\rho^{B,f}_t \circ \Ad(U_h)) (S_i^A) e_y^B
=  U_h \rho_t^{A,\Psi_h(f)}(S_i^A ) U_h^* e_y^B 
\end{equation*}
so that 
\begin{equation*}
\rho^{B,f}_t \circ \Ad(U_h)  
=\Ad(U_h) \circ  \rho_t^{A,\Psi_h(f)}
\quad \text{ for } t \in \RZ.
\end{equation*}
By setting 
$\Phi = \Ad(U_h): \OA\rightarrow \OB$,
we have the desired isomorphism.
\end{proof}
In the proof of \cite[Theorem 5.7]{MaPacific}
we see  that if there exists an isomorphism
$\Phi:\OA \rightarrow \OB$   
such that $\Phi(\DA) = \DB$,
we can take a homeomorphism
$h:X_A\rightarrow X_B$ which gives rise to a continuous orbit equivalence
between $(X_A,\sigma_A)$ and 
$(X_B,\sigma_B)$
and satisfies
$\Phi(f) = f \circ h^{-1}$
for $f \in \DA$.
We apply Theorem \ref{thm:COETFAE} to 
obtain the following theorem.
\begin{theorem} \label{thm:fg}
Let $f\in C(X_B,\Z),g \in C(X_A,\Z)$.
\begin{enumerate}
\renewcommand{\theenumi}{\roman{enumi}}
\renewcommand{\labelenumi}{\textup{(\theenumi)}}
\item
There exists an isomorphism
$\Phi:\OA\rightarrow \OB$ satisfying  
\begin{equation}
\Phi(\DA)=\DB
\quad
\text{ and }
\quad
\Phi\circ \rho^{A,g}_t = 
\rho^{B,f}_t\circ \Phi, \, \qquad t \in \RZ \label{eq:PhiAgBf}
\end{equation}
if and only if there exists a homeomorphism
$h: X_A\rightarrow X_B$ 
which gives rise to a continuous orbit equivalence
between $(X_A,\sigma_A)$ and $(X_B,\sigma_B)$
such that
$\Psi_h(f) =g$.
\item
There exist an 
isomorphism
$\Phi:\OA\rightarrow \OB$   
and
 a unitary one-cocycle 
$v_t, \, t \in \RZ$ 
in $\OA$ relative to
$\rho^{A,g}$
 satisfying  
\begin{equation}
\Phi(\DA) =\DB
\quad
\text{ and }
\quad
\Phi\circ \Ad(v_t) \circ \rho^{A,g}_t 
= 
\rho^{B,f}_t\circ \Phi, \, \qquad t \in \RZ
\label{eq:cocyclePhiAgBf}
\end{equation}
if and only if there exists a homeomorphism
$h: X_A\rightarrow X_B$ 
which gives rise to a continuous orbit equivalence
between $(X_A,\sigma_A)$ and $(X_B,\sigma_B)$
such that
$[\Psi_h(f)] =[g]$ in $H^A$.
\end{enumerate}
\end{theorem}
\begin{proof}
(i) 
The if part follows from Theorem \ref{thm:COETFAE}.
We will show the only if part.
Represent the algebras 
$\OA$ on $\HA$ and $\OB$ on $\HB$
as in the proof of Theorem \ref{thm:COETFAE}.
Suppose that there exists an isomorphism
$\Phi:\OA\rightarrow \OB$ satisfying
\eqref{eq:PhiAgBf}.
Take a homeomorphism
$h:X_A\rightarrow X_B$ which gives rise to a continuous orbit equivalence
between $(X_A,\sigma_A)$ 
and 
$(X_B,\sigma_B)$
and satisfies
$\Phi(f) = f \circ h^{-1}$
for $f \in \DA$.
Let $\varPhi_h: \OA \rightarrow \OB$
be the isomorphism $\Ad(U_h)$
defined in the proof of Theorem \ref{thm:COETFAE}
so that
\begin{equation}
\varPhi_h(\DA)=\DB
\quad
\text{ and }
\quad
\varPhi_h\circ \rho_t^{A,\Psi_h(f)} = 
\rho^{B,f}_t\circ \varPhi_h, \, \qquad t \in \RZ \label{eq:varPhiAgBf}.
\end{equation}
By the construction of $\varPhi_h$, we know that
$\varPhi_h(f) = f \circ h^{-1}$
for $f \in \DA$.
Hence the automorphism
$\alpha = \varPhi_h^{-1}\circ \Phi $ on $\OA$
satisfies $\alpha|_{\DA} = \id$.
By
\cite[Theorem 6.5 (1)]{MaPacific},
there exists a unitary one-cocycle
$V_\alpha$ in $\DA$ relative to $\alpha$ such that
$\alpha(S_\mu) = V_\alpha(k) S_\mu$ for $\mu \in B_k(X_A)$.
Hence
$\alpha(S_i) = V_\alpha(1)S_i$,
$V_\alpha(1) \in \DA$ 
so that
$\Phi(S_i) = \varPhi_h(V_\alpha(1)S_i)$ 
for $i=1,\dots,N$.
By \eqref{eq:PhiAgBf} and \eqref{eq:varPhiAgBf}, we see
\begin{equation*}
\Phi\circ \rho^{A,g}_t(S_i) 
= 
(\varPhi_h\circ \rho_t^{A,\Psi_h(f)}\circ
\varPhi_h^{-1}) \circ \Phi(S_i), \qquad t \in \RZ,
\, i=1,\dots, N. \label{eq:PhivarPhiAgBf}
\end{equation*}
The left hand side above equals
\begin{equation*}
\Phi(\rho^{A,g}_t(S_i)) 
= 
\Phi(U_t(g)S_i)
= 
\Phi(U_t(g)) \Phi(S_i).
\end{equation*}
The right hand side above equals
\begin{align*}
& (\varPhi_h\circ \rho_t^{A,\Psi_h(f)}\circ
\varPhi_h^{-1}) \circ \Phi(S_i) \\
= & \varPhi_h(V_\alpha(1) \rho_t^{A,\Psi_h(f)}(S_i)) 
=  \varPhi_h(U_t(\Psi_h(f)) ) \varPhi_h(V_\alpha(1) S_i)
=  \varPhi_h(U_t(\Psi_h(f)) ) \Phi(S_i). 
\end{align*}
Hence we have
$
\Phi(U_t(g)) \Phi(S_i)
= 
\varPhi_h(U_t(\Psi_h(f)) ) \Phi(S_i).
$
As both $U_t(g), U_t(\Psi_h(f))$
belong to $\DA$ and
$\Phi = \varPhi_h $ on $\DA$,
we obtain that 
$U_t(g) = U_t(\Psi_h(f))$
 and
hence $g =  \Psi_h(f)$.

(ii) 
We first show the if part.
Suppose that
$[\Psi_h(f)] =[g]$.
We may take $b \in C(X_A,\Z)$ such that 
$\Psi_h(f) = g + b -b\circ\sigma_A$.
By (i),  
there exists an isomorphism
$\Phi: \OA \rightarrow \OB$
such that
$\Phi(\DA) = \DB$
and
$
\Phi\circ \rho^{A,g+b-b\circ\sigma_A}_t 
= 
\rho^{B,f}_t\circ \Phi.
$
Since the  unitary 
$v_t = U_t(b)$
satisfies 
$
\Ad(v_t) \circ
\rho^{A,g}_t
=\rho^{A,g+b-b\circ\sigma_A}_t, 
$ 
the equality
\eqref{eq:cocyclePhiAgBf}
holds.

For the proof of the only if part, 
assume that there exist an 
isomorphism
$\Phi:\OA\rightarrow \OB$   
and
 a unitary one-cocycle 
$v_t, \, t \in \RZ$ 
in $\OA$ relative to
$\rho^{A,g}$
 satisfying the equality \eqref{eq:cocyclePhiAgBf}.
As $\rho^{A,g}_t|_{\DA} = \id$ and
$\rho^{B,f}_t|_{\DB} = \id$,
the equality \eqref{eq:cocyclePhiAgBf} 
implies that $\Phi(v_t)$ 
commutes with elements of 
$\Phi(\DA) = \DB$.
Since $\DB$ is maximal abelian in $\OB$,
the unitaries
$\Phi(v_t), t \in \RZ$ belong to $\DB$
and hence
$v_t, t \in \RZ$ belong to $\DA$
which satisfy
$\rho^{A,g}_s(v_t) = v_t$ for $s,t \in \RZ$.
The map $t \in \RZ \rightarrow v_t \in \DA$
then yields a unitary representation of $\RZ$.
Take $b \in C(X_A,\Z)$ 
such that 
$v_t = U_t(b)$
and hence 
 the equality
$
\Phi\circ \rho^{A,g+b-b\circ\sigma_A}_t 
= 
\rho^{B,f}_t\circ \Phi
$
follows.
By (i), there exists a homeomorphism
$h:X_A \rightarrow X_B$ which gives rise to a continuous orbit equivalence
between $(X_A,\sigma_A)$ and $(X_B,\sigma_B)$
such that
$\varPhi_h(f) = g+b-b\circ\sigma_A$
so that 
$[\Psi_h(f) ] =[g]$.
\end{proof}
Under the assumption of Theorem \ref{thm:COETFAE},
let us define the functions 
$c_1 \in C(X_A,\Z)
$ and
$ 
 c_2 \in C(X_B,\Z)
$
 by 
$c_1(x) = l_1(x) -k_1(x), x \in X_A$
and
$c_2(y) = l_2(y) -k_2(y), y \in X_B$,
respectively,
They are called the cocycle functions 
which play a key r{\^{o}}le in \cite{MaPre2014} and \cite{MMPre2014}.
They do not depend on the choices of the functions
$k_1, l_1$ and $k_2,l_2$ as long as they are satisfying 
\eqref{eq:orbiteq1x} and \eqref{eq:orbiteq2y},
respectively (\cite[Lemma 4.1]{MMPre2014}).
By the formula \eqref{eq:Psihfx},
we know $\Psi_h(1_B) = c_1$ and similarly 
$\Psi_{h^{-1}}(1_A) = c_2$,
where 
$1_A, 1_B$ are the constant functions on $X_A, X_B$ taking its values $1$,
respectively.
The actions 
$\rho_t^{A,1_A}$ for the constant function $1_A$
and
$\rho_t^{B,1_B}$ for the constant function $1_B$
are the gauge actions 
$\rho_t^A$ on $\OA$ and 
$\rho_t^B$ on $\OB$, respectively.
Theorem \ref{thm:fg} (i) implies the following corollary.
\begin{corollary}\label{coro:COEgauge}
The one-sided topological Markov shifts
$(X_A, \sigma_A)$ and $(X_B,\sigma_B)$
are continuously orbit equivalent
if and only if  there exist an isomorphism
$\Phi:\OA \rightarrow \OB$
and continuous functions
$c_1:X_A \rightarrow \Z,$
$c_2:X_B \rightarrow \Z,$
 such that 
\begin{equation*}
\Phi(\DA) = \DB
\quad 
\text{  and }
\quad
\Phi \circ \rho^{A}_t = \rho^{B,c_2}_t \circ \Phi,
\quad
\Phi^{-1} \circ \rho^{B}_t = \rho^{A,c_1}_t \circ \Phi
\quad
\text{ for }
 t \in \RZ.
\end{equation*}
\end{corollary}

In the case  $f=1_B, g=1_A$ 
in Theorem \ref{thm:fg},
the first assertion (i) describes
the commutativity of the gauge actions.
We say $(X_A,\sigma_A)$ and $(X_B,\sigma_B)$
 to be {\it eventually one-sided conjugate\/}
 if there exist a homeomorphism
$h:X_A\rightarrow X_B$ and continuous functions
$k_1:X_A\rightarrow \Z$
and
$k_2:X_B\rightarrow \Z$
such that
\begin{align}
\sigma_B^{k_1(x)} (h(\sigma_A(x))) 
& = \sigma_B^{k_1(x)+1}(h(x))
\quad 
\text{ for} \quad 
x \in X_A,  \label{eq:oneorbiteq1x} \\
\sigma_A^{k_2(y)} (h^{-1}(\sigma_B(y))) 
& = \sigma_A^{k_2(y)+1}(h^{-1}(y))
\quad 
\text{ for } \quad 
y \in X_B. \label{eq:oneorbiteq2y}
\end{align}  
If the conditions \eqref{eq:oneorbiteq1x} and \eqref{eq:oneorbiteq2y}
are fulfilled, 
one may take the functions  
$k_1, k_2$ to be the constant functions 
taking its values 
$K =\Max \{k_1(x), k_2(y) \mid x \in X_A, y \in X_B\}$.
As $\Psi_{h^{-1}} = (\Psi_h)^{-1}$,
the condition 
$\Psi_h(1_B) = 1_A$
is equivalent to
the condition 
$\Psi_{h^{-1}}(1_A) = 1_B$.
The latter and hence the former conditions
are equivalent to the conditions 
$l_1(x) = k_1(x) +1$ and 
$l_2(y) = k_2(y) +1$
by the  definition of $\Psi_h, \Psi_{h^{-1}}$. 
This means that the eventual one-sided conjugacy 
is equivalent to the condition
$\Psi_h(1_B) = 1_A$.
As a more general equivalence relation,
a notion of strongly continuous orbit equivalence 
between one-sided topological Markov shifts 
has been introduced in \cite{MaPre2014}
which is stronger than 
continuous orbit equivalence
but weaker than  eventual one-sided conjugacy. 
Two one-sided topological Markov shifts
$(X_A, \sigma_A)$ and $(X_B,\sigma_B)$
are said to be strongly continuous orbit equivalent
if there exists 
a homeomorphism
$h:X_A \rightarrow X_B$
giving  rise to a continuous orbit equivalence 
such that 
$[\Psi_h(1_B)] =[1_A]$ in $H^A$,
which is equivalent to the condition
$[\Psi_{h^{-1}}(1_A)] =[1_B]$ in $H^B$
(\cite[Lemma 4.3]{MaPre2014}).
We then have the following corollary of Theorem \ref{thm:fg}.
\begin{corollary}[{cf. \cite[2,17. Proposition]{CK},
\cite[Theorem 6.7]{MaPre2014}}] \label{coro:evstcoe}
Let $\rho^A$ (resp. $\rho^B$)
be the gauge action on $\OA$ (resp. $\OB$).
\begin{enumerate}
\renewcommand{\theenumi}{\roman{enumi}}
\renewcommand{\labelenumi}{\textup{(\theenumi)}}
\item
There exists an isomorphism
$\Phi:\OA \rightarrow \OB$ 
such that 
\begin{equation*}
\Phi(\DA) = \DB
\quad 
\text{  and }
\quad
\Phi \circ \rho^A_t = \rho^B_t \circ \Phi, 
\qquad t \in {\RZ}
\end{equation*}
if and only if 
$(X_A, \sigma_A)$ and $(X_B,\sigma_B)$
are eventually one-sided conjugate.
\item
There exist an isomorphism
$\Phi:\OA \rightarrow \OB$ 
and a unitary representation $v$ of $\RZ$ to $\DB$
such that 
\begin{equation*}
\Phi(\DA) = \DB
\quad 
\text{  and }
\quad
\Phi \circ \rho^A_t = \Ad(v_t)\circ \rho^B_t \circ \Phi, 
\qquad t \in {\RZ}
\end{equation*}
if and only if 
$(X_A, \sigma_A)$ and $(X_B,\sigma_B)$
are strongly continuous orbit equivalent.
\end{enumerate}
\end{corollary}
\begin{remark}
Cuntz and Krieger had proved that
if 
$(X_A, \sigma_A)$ and $(X_B,\sigma_B)$
are one-sided conjugate,
then there exists an isomorphism
$\Phi:\OA \rightarrow \OB$ 
such that 
$
\Phi(\DA) = \DB
$
and
$\Phi \circ \rho^A_t = \rho^B_t \circ \Phi, 
t \in {\RZ}
$
(\cite[2.17. Proposition]{CK}).
The converse implication is seen in \cite{Tomforde}
as an open question.
By the above corollary, the open question may be rephrased 
such that whether or not
the eventual one-sided conjugacy implies one-sided conjugacy.
This seems to be open.
\end{remark}

We denote by $C(X_A,\R)$ 
the set of real valued continuous functions
on $X_A$.
Let
$\gamma^{A,f}_r\in\Aut(\OA), r \in \R$ 
be the one-parameter automorphism for $f \in C(X_A,\R)$
on $\OA$
defined by
\begin{equation}
\gamma^{A,f}_r(S_i) = \exp(\sqrt{-1}r f)S_i, \qquad i=1,\dots,N. 
\label{eq:gammaArf}
\end{equation}
For a homeomorphism $h:X_A\rightarrow X_B$ 
which gives rise to a continuous orbit equivalence 
between 
$(X_A, \sigma_A)$
and 
$(X_B, \sigma_B)$,
we define a linear map 
$\Psi_h: C(X_B, \R)\rightarrow C(X_A,\R)$
by the same formula  as in \eqref{eq:Psihfx}.
We may show the following proposition
in a similar way to Theorem \ref{thm:COETFAE}. 
\begin{proposition}\label{prop:COETFAER}
If the one-sided topological Markov shifts
$(X_A, \sigma_A)$ and $(X_B,\sigma_B)$
are continuously orbit equivalent
via a homeomorphism $h:X_A\rightarrow X_B$,
then there exists an isomorphism
$\Phi:\OA \rightarrow \OB$ such that 
\begin{equation*}
\Phi(\DA) = \DB
\quad 
\text{  and }
\quad
\Phi \circ \gamma^{A,\Psi_h(f)}_r = \gamma^{B,f}_r \circ \Phi
\quad
\text{ for }
f \in C(X_B,\R), \,\, 
 r \in \R.
\end{equation*}
\end{proposition}

\section{Strong shift equivalence and circle actions}
It is well-known that 
a topological Markov shift
defined by a square matrix $A$ with entries in $\{0,1\}$
is naturally identified with 
a shift of finite type of the edge shift defined by 
the underlying directed graph. 
 In this section,
we consider edge shifts and hence  square matrices
with entries in nonnegative integers 
(cf.  \cite{Kitchens}, \cite{LM}, \cite{Williams}, etc.).
Such a matrix is simply called  a nonnegative square matrix.
Our hypothesis that the shift space $X_A$ is homeomorphic to 
a Cantor discontinuum forces the nonnegative square matrix
to be irreducible and not any permutation matrix.
For a nonnegative square matrix $A$,
denote by $E_A$ the edge set of the underlying graph $G_A$.
The two-sided shift space
$\bar{X}_A$  is identified with  
the two-sided sequences of concatenated edges of $E_A$.    
Two square matrices $A, B$ are said to be elementary equivalent 
if there exist  nonnegative rectangular matrices $C, D$ such that
$A = CD$ and $B=DC$ (\cite{Williams}, cf. \cite{LM}).
The graphs 
$G_A$ and $G_B$ become both bipartite graphs.
We may then take certain  bijections
$\varphi_{A,CD}$
from $E_A$ to a subset of $E_C \times E_D$
and 
$\varphi_{B,DC}$
from $E_B$ to a subset of $E_D \times E_C$.
We fix such bijections.
Through the bijections,
we may identify an edge $a$ of $E_A$ with a pair $cd$ of
edges $c \in E_C$ and $d\in E_D$,
and similarly
an edge $b$ of $E_B$ with a pair $dc$ of
edges $d \in E_D$ and $c\in E_C$.
An equivalence relation generated by elementary equivalences is called 
the strong shift equivalence. 
In \cite{Williams}, R. Williams  
has proved that two-sided topological Markov shifts
$(\bar{X}_A, \bar{\sigma}_A)$
and
$(\bar{X}_B, \bar{\sigma}_B)$
for irreducible matrices $A$ and $B$
are topologically conjugate if and only if 
the matrices $A$ and $B$ are strong shift equivalent.

For a $C^*$-algebra $\mathcal{A}$ without unit,
let $M({\mathcal{A}})$ stand for its multiplier $C^*$-algebra
defined by 
\begin{equation*}
M({\mathcal{A}}) = \{ a \in {\mathcal{A}}^{**} 
\mid a {\mathcal{A}} \subset {\mathcal{A}}, \, 
    {\mathcal{A}} a \subset {\mathcal{A}} \}. 
\end{equation*}
An action $\alpha$ of $\RZ$ on $\mathcal{A}$ always extends to 
  $M({\mathcal{A}})$ and we write the extended action still $\alpha$.
For an action $\alpha$ of $\RZ$ on ${\mathcal{A}}$,
a unitary one-cocycle $u_t, t \in \RZ$ relative to $\alpha$
is a continuous map $t \in \RZ \rightarrow u_t \in \U(M({\mathcal{A}}))$
to the unitary group  $\U(M({\mathcal{A}}))$ 
satisfying
$u_{t+s} = u_s \alpha_s(u_t), s,t \in \RZ$.
It is known that
if two matrices $A$ and $B$ are elementary equivalent,
there exists an isomorphism
$\Phi:\SOA \rightarrow \SOB$ 
and  a unitary one-cocycle
$u_t \in \U(M(\SOA))$ relative to the gauge action
$\rho^{A} \otimes\id$
such that 
\begin{equation*}
\Phi(\SDA) = \SDB
\quad
\text{ and }
\quad
\Phi \circ \Ad(u_t) \circ (\rho^{A}_t\otimes \id)
 = (\rho^{B}_t\otimes\id) \circ \Phi
\quad
\text{ for }
 t \in \RZ
\end{equation*}
(cf. \cite{CK}, \cite{MaETDS2004}, \cite{MPT}).
The proof given in \cite{MaETDS2004}
is due to a version of equivariant Morita equivalence 
(cf. \cite{BGR}, \cite{Combes}, \cite{CKRW}, \cite{RaeWill}).

In this section, 
we assume that two matrices $A$ and $B$ are elementary equivalent
such that
$A =CD$ and $B = DC$.
We set the square matrix
$Z =
\begin{bmatrix}
0 & C \\
D & 0
\end{bmatrix}
$
so that we see
\begin{equation*}
Z^2 =
\begin{bmatrix}
CD & 0 \\
0 & DC
\end{bmatrix}
=
\begin{bmatrix}
A & 0 \\
0 & B
\end{bmatrix}.
\end{equation*}
Let $\OZ$ be the Cuntz--Krieger algebra for the matrix $Z$.
Let us denote by
$ S_c, S_d, c \in E_C, d \in E_D$ 
the generating partial isometries of $\OZ$
such that 
$
\sum_{c \in E_C} S_c S_c^*
+
\sum_{d \in E_D} S_d S_d^*
=1$
and
\begin{equation*}
S_c^* S_c = \sum_{d \in E_D}Z(c,d) S_d S_d^*, \qquad
S_d^* S_d = \sum_{c \in E_C}Z(d,c) S_c S_c^*
\end{equation*}
for $c \in E_C, d \in E_D$.
We note that 
$S_c S_d \ne 0$ 
(resp. $S_d S_c \ne 0$) 
if and only if 
$\varphi_{A,CD}(a) = cd$
(resp.
$\varphi_{B,DC}(b) = dc$)
 for some $a \in E_A$
(resp. $b \in E_B$).
In this case, we identify 
$cd$ (resp. $dc$) 
with $a$ (resp. $b$)
through the map 
$\varphi_{A,CD}$
(resp. $\varphi_{B,DC}$).
We then write
$S_{cd} = S_a$
(resp. $S_{dc} = S_b$)
where $S_{cd} = S_c S_d$ (resp. $S_{dc} = S_d S_c$)
belongs to $\OZ$
whereas 
$S_a$ (resp. $S_b$) belongs to $\OA$ (resp. $\OB$).
Put
$P_C = \sum_{c \in E_C} S_c S_c^*$
and
$P_D = \sum_{d \in E_D} S_d S_d^*$
so that $P_C + P_D =1$.
It has been shown in \cite{MaETDS2004}
that 
\begin{equation}
P_C \OZ P_C = \OA, \qquad P_D \OZ P_D = \OB,\qquad
 \DZ P_C = \DA, \qquad  \DZ P_D = \DB
\end{equation}
and 
$P_C \OZ P_D$ has a natural structure of 
$\OA-\OB$ imprimitivity bimodule (\cite{Rieffel1}, \cite{Rieffel2}).
As $\DZ$ is commutative, we note that
$
P_C \DZ P_C =\DZ P_C
$
and
$P_D \DZ P_D =\DZ P_D.
$

We take and fix a function $f \in C(X_A,\Z)$.
It is regarded as an element of 
$\DA$ and hence of $\DZ$
by identifying it with 
$f\oplus 0$ in $\DA \oplus \DB = \DZ$.
Since
$
 \exp{(2\pi\sqrt{-1}t (f\oplus 0))}
 =\exp{(2\pi\sqrt{-1}t f)} +P_D
 \in \U(\DZ)$,
the automorphism
$\rho^{Z,f}_t \in \Aut(\OZ)$
for $t \in \RZ$ 
satisfies
\begin{equation}
\rho^{Z,f}_t(S_c) = \exp{(2\pi\sqrt{-1}t f)}S_c
\quad \text{ for } c \in E_C,\qquad
\rho^{Z,f}_t(S_d) = S_d
\quad \text{ for } d \in E_D.
\end{equation}
For $a \in E_A, b \in E_B$ such that
$
\varphi_{A,CD}(a) = cd,
\varphi_{B,DC}(b) = dc,$
we then have
\begin{align*}
\rho^{Z,f}_t(S_c S_d) 
& =\rho^{Z,f}_t(S_c)\rho^{Z,f}_t(S_d) 
 =\exp{(2\pi\sqrt{-1}t f)} S_c S_d 
 =\rho^{A,f}_t(S_{a}), \\ 
\rho^{Z,f}_t(S_d S_c) 
& =\rho^{Z,f}_t(S_d)\rho^{Z,f}_t(S_c) 
 = S_d \exp{(2\pi\sqrt{-1}t f)} S_c  
 = S_d \exp{(2\pi\sqrt{-1}t f)} S_d^* S_{b}.
\end{align*} 
Put 
$\phi(f) = \sum_{d\in E_D}S_d f S_d^* \in \DZ$.
Since 
$P_D \phi(f) P_D = \phi(f)$,
we see that
$\phi(f) \in \DB$ and hence $\phi(f) \in C(X_B,\Z)$ 
such that 
\begin{equation*}
\sum_{d \in E_D}S_d \exp{(2\pi\sqrt{-1}t f)}S_d^*
=
\exp{(2\pi\sqrt{-1}t\phi(f))} \in \U(\DB).
\end{equation*}
Therefore we see
\begin{equation*}
\rho^{Z,f}_t(S_d S_c)  
=\exp{(2\pi\sqrt{-1}t \phi(f))} S_{b}
=\rho^{B,\phi(f)}_t (S_{b}).
\end{equation*}
We have thus proved the following lemma.
\begin{lemma}
Suppose that 
$A = CD$ and $B =DC$.
For $f \in C(X_A,\Z)$,
there exists 
$\phi(f) \in C(X_B,\Z)$
such that 
\begin{equation}
\rho^{Z,f}_t(S_c S_d) =\rho^{A,f}_t(S_{a}), \qquad
\rho^{Z,f}_t(S_d S_c) =\rho^{B,\phi(f)}_t(S_{b}),
\qquad t \in \RZ \label{eq:rhozf}
\end{equation}
where
$a \in E_A, b \in E_B$
and
$c \in E_C, d\in E_D$
are satisfying
$
\varphi_{A,CD}(a) = cd,
\varphi_{B,DC}(b) = dc$.
\end{lemma}
We note that the following symmetric equations to \eqref{eq:rhozf} 
\begin{equation}
\rho^{Z,g}_t(S_d S_c) =\rho^{B,g}_t(S_{b}), \qquad
\rho^{Z,g}_t(S_c S_d) =\rho^{A,\psi(g)}_t(S_{a}),
\qquad t \in \RZ \label{eq:rhozg}
\end{equation}
for 
$g \in C(X_B,\Z)$ and
$\psi(g) =\sum_{c \in E_C} S_c g S_c^* \in C(X_A,\Z)$
hold.
We then have the following lemma.
\begin{lemma}
Suppose that 
$A = CD$ and $B =DC$.
The homomorphisms
$\phi:C(X_A,\Z) \rightarrow  C(X_B,\Z)$
and
$\psi:C(X_B,\Z) \rightarrow  C(X_A,\Z)$
 defined by
\begin{equation*}
\phi(f) =\sum_{d \in E_D} S_d f S_d^*,\qquad
\psi(g) =\sum_{c \in E_C} S_c g S_c^* 
\end{equation*}
satisfy the following equalities
\begin{equation}
(\psi \circ \phi)(f) = f \circ \sigma_A,\qquad
(\phi \circ \psi)(g) = g \circ \sigma_B
\end{equation}
for $f \in C(X_A,\Z)$ and $g \in C(X_B,\Z)$.
Hence they induce isomorphisms
$\bar{\phi}:H^A\rightarrow H^B$
and
$\bar{\psi}:H^B\rightarrow H^A$
of ordered abelian groups in a natural way.
\end{lemma}
\begin{proof}
We have
\begin{equation*}
(\psi \circ \phi)(f) 
 = \sum_{c \in E_C} S_c (\sum_{d \in E_D} S_d f S_d^*) S_c^*
 = \sum_{a \in E_A} S_{a} f S_{a}^* 
= f \circ \sigma_A 
\quad
\text{ for } f \in C(X_A,\Z)
\end{equation*}
and
similarly 
$(\phi \circ \psi)(g) = g\circ \sigma_B
$
for
$ g \in C(X_B,\Z).$
It is easy to see that
the equality
$\phi(f - f \circ \sigma_A) = \phi(f) - \phi(f) \circ \sigma_B$
holds
so that 
$\phi:C(X_A,\Z) \rightarrow C(X_B,\Z)$ induces 
an isomorphism written
$\bar{\phi}:H^A\rightarrow H^B$
of ordered abelian groups whose inverse is 
$\bar{\psi}:H^B\rightarrow H^A$
induced by 
$\psi: C(X_B,\Z) \rightarrow C(X_A,\Z)$.
\end{proof}
We can prove the following proposition.
\begin{proposition}
Suppose that 
$A = CD$ and $B =DC$.
Then there exists an isomorphism
$\Phi:\SOA \rightarrow \SOB$ satisfying 
$\Phi(\SDA) = \SDB$
such that
for each function
$f\in C(X_A,\Z)$
there exists a unitary one-cocycle
$u_t^f \in \U(M(\SOA))$  relative to 
$\rho^{A,f} \otimes\id$
such that 
\begin{equation*}
\Phi \circ \Ad(u_t^f) \circ (\rho^{A,f}_t\otimes \id)
 = (\rho^{B,\phi(f)}_t\otimes\id) \circ \Phi
\quad
\text{ for }
 t \in \RZ.
\end{equation*}
\end{proposition}
\begin{proof}
By \cite[Proposition 4.1]{MaETDS2004},
there exist
partial isometries
$v_{A}, v_{B} \in M(\SOZ)$
such that
\begin{equation}
v_{A}^*v_{A} = v_{B}^*v_{B} =1,
\qquad 
v_{A}v_{A}^* = P_C \otimes 1,
\qquad
v_{B}v_{B}^* = P_D \otimes 1
\end{equation}
and
\begin{equation}
\Ad(v_{A}^*):\SOA \rightarrow \SOZ
\quad
\text{ and }
\quad
\Ad(v_{B}^*):\SOB \rightarrow \SOZ
\end{equation}
are isomorphisms satisfying
\begin{equation*}
\Ad(v_{A}^*)(\SDA) = \SDZ
\quad
\text{ and }
\quad
\Ad(v_{B}^*)(\SDB) = \SDZ.
\end{equation*}
Put
$w = v_{B}v_{A}^* \in M(\SOZ)$
and
$\Phi =\Ad(w): \SOA\rightarrow\SOB$
which satisfies
$\Phi(\SDA) = \SDB$.
For
$f\in C(X_A,\Z)$,
we define
$u_t^f \in \U(M(\SOA))$ 
by setting
$u_t^f = w^*(\rho^{Z,f}_t \otimes\id)(w)$,
where 
$\rho^{Z,f}_t \otimes\id$
is extended to an automorphism of $M(\SOZ)$.
As 
$(P_C \otimes 1) v_{A} = v_A$, 
we have
\begin{equation*}
u_t^f = (P_C \otimes 1) w^*
        (\rho^{Z,f}_t \otimes\id)(w)(P_C \otimes 1) 
\end{equation*}
so that
$u_t^f \in (P_C\otimes 1)M(\SOZ)(P_C\otimes 1)$.
Since
$(P_C\otimes 1)M(\SOZ)(P_C\otimes 1) =M(\SOA)$,
we have
$u_t^f \in M(\SOA)$.
As $\rho^{A,f}_t(T) =\rho^{Z,f}_t(T)$ for $T \in \OA$,
it then follows that
for $T\otimes K\in \SOA$
\begin{align*}
[\Ad(u_t^f) \circ (\rho^{A,f}_t\otimes \id)](T\otimes K)
= & u_t^f(\rho^{A,f}_t(T)\otimes  K) (u_t^f)^* \\
= &  w^*(\rho^{Z,f}_t \otimes\id)(w)(\rho^{Z,f}_t(T)\otimes  K) 
(w^*(\rho^{Z,f}_t \otimes\id)(w))^* \\
= &  w^*(\rho^{Z,f}_t \otimes\id)(w (T\otimes  K)w^*) w. 
\end{align*}
Since
$w (T\otimes  K)w^* \in \SOB$
and $\rho^{Z,f}_t(R) =\rho^{B,\phi(f)}_t(R)$ for $R \in \OB$,
we have
\begin{align*}
w^*(\rho^{Z,f}_t \otimes\id)(w (T\otimes  K)w^*) w
= & w^*(\rho^{B,\phi(f)}_t \otimes\id)(w (T\otimes  K)w^*) w \\
= & [\Phi^{-1}\circ (\rho^{B, \phi(f)}_t \otimes\id)
    \circ \Phi](T\otimes K).
\end{align*} 
Hence we obtain
\begin{equation*}
\Ad(u_t^f) \circ (\rho^{A,f}_t\otimes \id)
=  \Phi^{-1}\circ (\rho^{B,\phi(f)}_t \otimes\id)\circ \Phi.
\end{equation*} 
\end{proof}
\begin{corollary}\label{cor:SSE}
Suppose that 
$A$ and $B$ are strong shift equivalent.
Then there exist an isomorphism
$\Phi:\SOA \rightarrow \SOB$ satisfying 
$\Phi(\SDA) = \SDB$
and a homomorphism
$\phi:C(X_A,\Z) \rightarrow C(X_B,\Z)$
of ordered groups
which induces an isomorphism
between
$H^A$ and $H^B$ of ordered groups
such that
for each function
$f\in C(X_A,\Z)$
there exists a unitary one-cocycle
$u_t^f \in \U(M(\SOA))$  
relative to 
$\rho^{A,f} \otimes\id$
satisfying 
\begin{equation*}
\Phi \circ \Ad(u_t^f) \circ (\rho^{A,f}_t\otimes \id)
 = (\rho^{B,\phi(f)}_t\otimes\id) \circ \Phi
\quad
\text{ for }
 t \in \RZ.
\end{equation*}
\end{corollary}
We note that the above isomorphism
from $H^A$ to $H^B$ induced  by $\phi$ 
sends the class
$[1_A]$ of the constant $1$ function on $X_A$
to the class
$[1_B]$ of the constant $1$ function on $X_B$.

\section{Flow equivalence and circle actions}
For an $N \times N$ matrix
$A = [A(i,j)]_{i,j=1}^N$
with entries in $\{0,1\}$,
put
\begin{equation}
\tilde{A} =
\begin{bmatrix}
0      & A(1,1) & \cdots & A(1,N) \\
1      & 0      & \cdots & 0      \\
0      & A(2,1) & \cdots & A(2,N) \\
\vdots & \vdots &        & \vdots \\
0      & A(N,1) & \cdots & A(N,N)
\end{bmatrix}, \label{eq:tildeA}
\end{equation}
which is called the expansion of $A$ at the vertex $1$.
Let 
$\{0,1,\dots,N\}$
be the set of symbols for the topological Markov shifts $(X_{\tilde{A}},\sigma_{\tilde{A}})$
for the matrix $\tilde{A}$.
Let us denote by 
$\tilde{S}_0, \tilde{S}_1,\dots, \tilde{S}_N$
the canonical generating partial isometries of the Cuntz--Krieger algebra
${\mathcal{O}}_{\tilde{A}}$
which satisfy
$\sum_{j=0}^N
\tilde{S}_j\tilde{S}_j^* =1,
\tilde{S}_i^*\tilde{S}_i
= \sum_{j=0}^N
\tilde{A}(i,j) \tilde{S}_j\tilde{S}_j^*
$ 
for $i=0,1,\dots,N$.
Put
$
P = \sum_{i=1}^N
\tilde{S}_i\tilde{S}_i^*.
$
We note that the identities 
\begin{equation}
 \tilde{S}_1^* P \tilde{S}_1 
=\tilde{S}_1^* \tilde{S}_1 = \tilde{S}_0 \tilde{S}_0^*,\qquad
P + \tilde{S}_0 \tilde{S}_0^* =P + \tilde{S}_1^* P \tilde{S}_1 =1
\label{eq:PtildeS}
\end{equation}
hold.
We decompose the set $\N$ of natural numbers such as 
$\N = \cup_{j=1}^\infty \N_j$ and 
$\N_j = \cup_{k=0}^\infty \N_{j_k}$ as disjoint unions such that 
$\N_{j_k}$ is an infinite set for each $k, j$.  
We take and fix a canonical system  
$e_{i,j},\,  i,j \in \N$
of matrix units which generates the algebra
${\mathcal{K}}$
such that $e_{i,i}, i\in \N$ are projections of rank one in $\mathcal{C}$.
We put the projections
$
f_j = \sum_{i \in\N_j} e_{i,i}, \,
f_{j_k} = \sum_{i \in\N_{j_k}} e_{i,i}.
$
Take a partial isometry $s_{j_k,j}$ such that
$s_{j_k,j}^* s_{j_k,j} = f_j, \, s_{j_k,j} s_{j_k,j}^* = f_j.$  
\begin{lemma}
There exists a partial isometry
$w_n \in M(\SOTA),\, n=1,2,\dots $
such that 
\begin{gather*}
w_n^* w_n = 1_{\tilde{A}}\otimes f_n, \qquad
w_n w_n^* \le P \otimes f_n, \\
w_n^*({\mathcal{D}}_{\tilde{A}} P\otimes{\mathcal{C}})w_n \subset
 \SDTA,
 \qquad
 w_n(\SDTA)w_n^* \subset {\mathcal{D}}_{\tilde{A}} P\otimes{\mathcal{C}}.   
\end{gather*}
\end{lemma}
\begin{proof}
Put
$w_n = P \otimes s_{n_0,n} + P \tilde{S}_1\otimes s_{n_1,n}$.
It then follows that
\begin{equation*}
w_n^* w_n 
=P \otimes f_n + \tilde{S}_1^*P \tilde{S}_1\otimes f_n
=(P+ \tilde{S}_1^*P \tilde{S}_1)\otimes f_n
= 1 \otimes f_n.
\end{equation*}
We note that 
$P \tilde{S}_1 = \tilde{S}_1$ 
and hence
$P \tilde{S}_1 P  = \tilde{S}_1 P =0$ 
 so that  
\begin{equation*}
w_n w_n^*  
 =P \otimes f_{n_0} +  P \tilde{S}_1\tilde{S}_1^* P\otimes f_{n_1} 
\le P \otimes (f_{n_0} + f_{n_1}) \le P\otimes f_n.
\end{equation*}
For $F \otimes L \in \SDTA$ with $F \in \DTA$ and $L \in {\mathcal{C}}$,
we have $P \tilde{S}_1 F P  = P F\tilde{S}_1^* P =0$ 
so that 
\begin{align*}
 w_n(F\otimes L)w_n^* 
=&  P F P \otimes s_{n_0,n} L s_{n_0,n}^*  
+ P \tilde{S}_1 F \tilde{S}_1^* P \otimes s_{n_1,n} L s_{n_1,n}^*, \\
w_n^*(P\otimes 1)(F\otimes L)(P\otimes 1)w_n
= & P F P \otimes  s_{n_0,n}^* L s_{n_0,n} 
+ \tilde{S}_1^* P F P \tilde{S}_1 \otimes s_{n_1,n}^*  L s_{n_1,n}
\end{align*}
so that 
$
 w_n(\SDTA)w_n^* \subset {\mathcal{D}}_{\tilde{A}} P\otimes{\mathcal{C}}
$
and
$
w_n^*({\mathcal{D}}_{\tilde{A}} P\otimes{\mathcal{C}})w_n \subset
 \SDTA.
$
\end{proof}
Let us denote by  $f_{n,m}$  a partial isometry satisfying 
$f_{n,m}^*f_{n,m}=f_m, f_{n,m}f_{n,m}^*=f_n.$
We set
$v_1 = w_1,\, v_{2} = (P\otimes f_1- v_{1} v_{1}^*)(P\otimes f_{1,2})$
and
\begin{equation*}
v_{2n-1} = w_n (1\otimes f_n- v_{2n-2}^* v_{2n-2}) , \quad
v_{2n} = (P\otimes f_n- v_{2n-1} v_{2n-1}^*)(P\otimes f_{n,n+1})
\quad \text{ for }
1< n \in \N.
\end{equation*}
We know that $\sum_{n=1}^\infty v_n$ converges to a partial isometry 
$v$ in the multiplier algebra $M(\SOTA)$ under the strict topology. 
By using the above lemma,
we then obtain the following lemma in a similar fashion to an argument 
giving in the proofs of 
\cite[Lemma 2.5]{Brown}
and 
\cite[Proposition 4.1]{MaETDS2004}.
\begin{lemma}
The partial isometry
$v \in M(\SOTA)$
satisfy  
\begin{gather*}
v^* v = 1_{\tilde{A}}\otimes 1, \qquad
v v^* = P \otimes 1, \\
v^*({\mathcal{D}}_{\tilde{A}} P\otimes{\mathcal{C}})v =
 \SDTA,
 \qquad
 v(\SDTA)v^* = {\mathcal{D}}_{\tilde{A}} P\otimes{\mathcal{C}}.   
\end{gather*}
\end{lemma}
We thus have
\begin{proposition}\label{prop:Phit}
Let $\tilde{A}$ be the matrix defined in \eqref{eq:tildeA}
for a matrix $A$.
Then there exists an isomorphism
$\tilde{\Phi}:\SOTA 
\rightarrow 
P{\mathcal{O}}_{\tilde{A}}P \otimes{\mathcal{K}}$ 
satisfying 
$\tilde{\Phi}(\SDTA) = 
{\mathcal{D}}_{\tilde{A}}P \otimes{\mathcal{C}}$
such that for each function
$\f\in C(X_{\tilde{A}},\Z)$
there exists a unitary one-cocycle
 $v_t^{\f} \in \U(M(\SOTA))$ relative to 
$\rho^{\tilde{A},\f} \otimes\id$
such that 
\begin{equation*}
\tilde{\Phi} \circ \Ad(v_t^{\f}) \circ (\rho^{\tilde{A},{\f}}_t\otimes \id)
 = (\rho^{\tilde{A},{\f}}_t\otimes\id) \circ \tilde{\Phi}
\quad
\text{ for }
 t \in \RZ.
\end{equation*}
\end{proposition}
\begin{proof}
Take a partial isometry
$v \in M(\SOTA)$
as in the above lemma.
Define 
$\tilde{\Phi}:\SOTA \rightarrow 
P{\mathcal{O}}_{\tilde{A}}P \otimes{\mathcal{K}}
$
by setting
$\tilde{\Phi}(T\otimes K) = v (T \otimes K)v^*$ for $T \in \OTA$
and
$K\in \mathcal{K}$.
We set 
$v_t^{\f}  = v^*(\rho_t^{\tilde{A},\f}\otimes \id) (v)$ 
for $t \in \RZ$.
It then follows that
\begin{align*}
(\rho^{\tilde{A},\f}_t\otimes \id) \circ \tilde{\Phi}(T\otimes K) 
& = (\rho^{\tilde{A},\f}_t\otimes \id) (v (T \otimes K)v^*) \\ 
& = (\rho^{\tilde{A},\f}_t\otimes \id) 
((P\otimes 1)v (T \otimes K)v^*(P\otimes 1)) \\ 
& = (P\otimes 1)(\rho^{\tilde{A},\f}_t\otimes \id) (v (T \otimes K) v^*)
    (P\otimes 1) \\ 
& = v v^*(\rho^{\tilde{A},\f}_t\otimes \id) (v) 
         (\rho^{\tilde{A},\f}_t\otimes \id)(T \otimes K)
         (\rho^{\tilde{A},\f}_t\otimes \id)(v^*) v v^* \\ 
& = \tilde{\Phi}\circ \Ad(v_t^{\f})\circ 
(\rho^{\tilde{A},\f}_t\otimes \id)(T \otimes K).
\end{align*}
Hence we have
\begin{equation*}
\tilde{\Phi} \circ \Ad(v_t^{\f}) \circ (\rho^{\tilde{A},\f}_t\otimes \id)
 = (\rho^{\tilde{A},{\f}}_t\otimes\id) \circ \tilde{\Phi}
\quad
\text{ for }
 t \in \RZ.
\end{equation*}
\end{proof}

Let
$\xi:X_A \rightarrow X_{\tilde{A}}$
be the continuous map
defined by
substituting  the symbol $1$ by the word $(1,0)$ 
in $X_A$ such as
\begin{equation}
(1,2,1,1,2,1,1,1,2, \dots) \in X_A
\rightarrow
(1,0,2,1,0,1,0,2,1,0,1,0,1,0,2, \dots) \in X_{\tilde{A}}.
\end{equation}
We define continuous functions
$k_1, l_1 \in C(X_A,\Zp)$
by setting
\begin{equation*}
k_1(x)=0,
\qquad
l_1(x)
=
\begin{cases}
2 & \text{ if } x_1 = 1, \\
1 & \text{ otherwise}
\end{cases}
\end{equation*}
for $x = (x_n)_{n \in \N} \in X_A$
so that 
$\xi:X_A \rightarrow X_{\tilde{A}}$
satisfies the identity
\begin{equation}
\sigma_{\tilde{A}}^{k_1(x)}(\xi(\sigma_A(x))) 
= \sigma_{\tilde{A}}^{l_1(x)}(\xi(x))
\quad \text{ for }
\quad 
x \in X_A. \label{eq: sigmaxi}
\end{equation}
Hence 
$\xi:X_A \rightarrow X_{\tilde{A}}$
is a continuous orbit map in the sense of 
\cite{MaPre2014}.
We define $\Psi_\xi(\f) \in C(X_A,\Z)$
for $\f \in C(X_{\tilde{A}},\Z) $
in the same formula as
\eqref{eq:Psihfx} so that
\begin{equation*}
\Psi_\xi(\f)(x) =
\begin{cases}
\f(\xi(x)) + \f(\sigma_{\tilde{A}}(\xi(x))) & \text{ if } x_1 =1, \\ 
\f(\xi(x)) & \text{ otherwise}
\end{cases}
\end{equation*}
for $x=(x_n)_{n \in \N} \in X_A$.
As in the proof of \cite[4.1 Theorem]{CK},
the partial isometries
$
\tilde{S}_1 \tilde{S}_0,
\tilde{S}_2, \dots,
\tilde{S}_N
$
satisfy the relations
\eqref{eq:CK}
so that we may regard them as the generating 
partial isometries
$S_1,S_2, \dots, S_N$
of $\OA$ and put
$
S_1 =\tilde{S}_1 \tilde{S}_0,
S_2 =\tilde{S}_2, \dots,
S_N =\tilde{S}_N.
$
\begin{lemma}[{cf. \cite[4.1 Theorem]{CK}}] \label{lem:Phi0A}
The identities
\begin{equation}
P \tilde{S}_1 \tilde{S}_0 P = S_1,
\qquad
P \tilde{S}_i P = S_i
\qquad
\text{ for } i=2,\dots,N
\end{equation}
hold and the correspondence
$
\Phi^0(P \tilde{S}_1 \tilde{S}_0 P) = S_1,
\Phi^0(P \tilde{S}_i P) = S_i, i=2,\dots,N
$
gives rise to an isomorphism 
$\Phi^0:P\OTA P \rightarrow \OA$
such that
$\Phi^0(\DTA P)= \DA$
and
\begin{equation*}
\Phi^0 \circ \rho^{{\tilde{A}},\f}_t 
=\rho^{A, \Psi_\xi(\f)}_t\circ \Phi^0
\quad
\text{ for }
\f \in C(X_{\tilde{A}},\Z),
t \in \RZ.
\end{equation*}  
\end{lemma}
\begin{proof}
Since
$P \tilde{S}_1 = \tilde{S}_1$
and
$
\tilde{S}_0 P = \tilde{S}_0,
$
we have
$P \tilde{S}_1\tilde{S}_0 P = {S}_1$.
It is easy to see that
$P \tilde{S}_i P = S_i$
for $i=2,\dots,N$.
As in the proof of \cite[4.1 Theorem]{CK}, 
the correspondence
$\Phi^0:P\OTA P \rightarrow \OA$ above
is shown to be an isomorphism satisfying 
$\Phi^0(\DTA P) = \DA$.
We note that 
$
U_t({\f}^2)\tilde{S}_1 \tilde{S}_0 
= U_t(\Psi_\xi(\f))S_1
$
which implies that
$
\Phi^0(U_t({\f}^2)\tilde{S}_1 \tilde{S}_0 )
= U_t(\Psi_\xi(\f))S_1.
$
It then follows that
\begin{align*}
\Phi^0 \circ \rho^{{\tilde{A}},\f}_t(P \tilde{S}_1 \tilde{S}_0 P)
&=\Phi^0(\rho^{{\tilde{A}},\f}_t( \tilde{S}_1 \tilde{S}_0 ))
=\Phi^0 (U_t({\f}^2)\tilde{S}_1 \tilde{S}_0 )\\
&= U_t(\Psi_\xi(\f))S_1 
 = \rho^{A,\Psi_\xi(\f)}_t(\Phi^0(P \tilde{S}_1 \tilde{S}_0 P)).
\end{align*}
Similarly
\begin{equation*}
\Phi^0 \circ \rho^{{\tilde{A}},\f}_t(P \tilde{S}_i P)
 = \rho^{A,\Psi_\xi(\f)}_t(\Phi^0(P \tilde{S}_i P)),
 \qquad 
 i=2,\dots,N.
\end{equation*}
Hence we obtain that 
$\Phi^0 \circ \rho^{{\tilde{A}},\f}_t 
=\rho^{A,\Psi_\xi(\f)}_t\circ \Phi^0$.
\end{proof}
Let 
$\eta:X_{\tilde{A}}\rightarrow X_A$
be the continuous map 
defined by deleting the symbol $0$ in elements of $X_{\tilde{A}}$
such as 
\begin{align*}
(1,0,2,1,0,1,0,2,1,0,1,0,1,0,2, \dots) \in X_{\tilde{A}}
&\rightarrow
(1,2,1,1,2,1,1,1,2, \dots) \in X_A, \\
(0,2,1,0,1,0,2,1,0,1,0,1,0,2,1, \dots) \in X_{\tilde{A}}
&\rightarrow
(2,1,1,2,1,1,1,2,1, \dots) \in X_A.
\end{align*}
We define continuous functions
$\tilde{k}_1, \tilde{l}_1 \in C(X_{\tilde{A}},\Zp)$
by setting
\begin{equation*}
\tilde{k}_1(\tilde{x})=0,
\qquad
\tilde{l}_1(\tilde{x})
=
\begin{cases}
0 & \text{ if } \tilde{x}_1 = 0, \\
1 & \text{ otherwise}
\end{cases}
\end{equation*}
for $\tilde{x} = (\tilde{x}_n)_{n \in \N} \in X_{\tilde{A}}$
so that 
$\eta:X_{\tilde{A}} \rightarrow X_A$
satisfies the identity
\begin{equation}
\sigma_A^{\tilde{k}_1(\tilde{x})}(\eta(\sigma_{\tilde{A}}(\tilde{x}))) 
= \sigma_A^{\tilde{l}_1(\tilde{x})}(\eta(\tilde{x}))
\quad \text{ for }
\quad 
\tilde{x} \in X_{\tilde{A}}. \label{eq: sigmatilde}
\end{equation}
Hence 
$\eta:X_{\tilde{A}} \rightarrow X_A$
is a continuous orbit map in the sense of 
\cite{MaPre2014}.
We define $\Psi_\eta(f)\in C(X_{\tilde{A}},\Z)$
for $f \in C(X_A,\Z)$
in the same formula as
\eqref{eq:Psihfx} so that
\begin{equation*}
\Psi_\eta(f)(\tilde{x}) =
\begin{cases}
0 & \text{ if } \tilde{x}_1 =0,\\
f(\eta(\tilde{x})) & \text{ otherwise}
\end{cases}
\end{equation*}
for $\tilde{x}=(\tilde{x}_n)_{n \in \N} \in X_{\tilde{A}}$.
We then have
\begin{lemma}\label{lem:Psietaxi}
The homomorphisms
$\Psi_\eta: C(X_A,\Z) \rightarrow C(X_{\tilde{A}},\Z)$
and
$\Psi_\xi: C(X_{\tilde{A}},\Z) \rightarrow C(X_A,\Z)$
induce isomorphisms of ordered abelian groups 
$\bar{\Psi}_\eta: H^A \rightarrow H^{\tilde{A}}$
and
$\bar{\Psi}_\xi: H^{\tilde{A}} \rightarrow H^A$,
respectively such that
they are inverses to each other.
\end{lemma}
\begin{proof}
Both the maps 
$\eta:X_{\tilde{A}} \rightarrow X_A$
and
$\xi:X_A \rightarrow X_{\tilde{A}}$
are continuous orbit maps
so that the homomorphisms
$\Psi_\eta: C(X_A,\Z) \rightarrow C(X_{\tilde{A}},\Z)$
and
$\Psi_\xi: C(X_{\tilde{A}},\Z) \rightarrow C(X_A,\Z)$
induce homomorphisms of ordered abelian groups 
$\bar{\Psi}_\eta: H^A \rightarrow H^{\tilde{A}}$
and
$\bar{\Psi}_\xi: H^{\tilde{A}} \rightarrow H^A$,
respectively (\cite{MaPre2014}).
It suffices to prove that they are inverses to each other.
It is easy to see that
\begin{align*}
\eta \circ \xi (x) 
&= x \quad \text{ for } x \in X_A, \\
\xi \circ \eta(\tilde{x})
& = 
{
\begin{cases}
\sigma_{\tilde{A}}(\tilde{x}) & \text{ if } \tilde{x}_1 =0, \\
\tilde{x} & \text{ otherwise }
\end{cases}
}
\quad
\text{ for }
\tilde{x} = (\tilde{x}_n)_{n \in \N} \in X_{\tilde{A}}.
\end{align*}
As in \cite[Proposition 3.4]{MaPre2014},
$\eta \rightarrow \Psi_\eta$ is contravariantly functorial,
so that we have   
\begin{align*}
(\Psi_\xi \circ \Psi_\eta) (f) 
&= \Psi_{\eta \circ \xi} (f) = f 
 \quad \text{ for } f \in C(X_A,\Z), \\
(\Psi_\eta \circ \Psi_\xi )(\f) (\tilde{x})
& =\Psi_{\xi \circ \eta}(\f)(\tilde{x})
 = 
{
\begin{cases}
0 & \text{ if } \tilde{x}_1 =0, \\
\f(\tilde{x}) + \f(\sigma_{\tilde{A}}(\tilde{x}))
 & \text{ if } \tilde{x}_1 =1, \\
\f(\tilde{x}) &  \text{ otherwise }
\end{cases}
}
\end{align*}
for 
$\tilde{x} = (\tilde{x}_n)_{n \in \N} \in X_{\tilde{A}}$.
We set
\begin{equation*}
\f_0(\tilde{x}) =
\begin{cases}
\f(\tilde{x}) & \text{ if } \tilde{x}_1 =0, \\
0  & \text{ otherwise.}
\end{cases}
\end{equation*}
It is straightforward to see that
$$
(\Psi_\eta \circ \Psi_\xi )(\f) (\tilde{x}) -\f(\tilde{x})
= \f_0(\sigma_{\tilde{A}}(\tilde{x})) - \f_0(\tilde{x})
\quad 
\text{ for  }
\tilde{x} \in X_{\tilde{A}}
$$
so that
$
(\Psi_\eta \circ \Psi_\xi )(\f) -\f
= \f_0\circ\sigma_{\tilde{A}} - \f_0
$
and hence
$
[(\Psi_\eta \circ \Psi_\xi )(\f)]
= [\f]
$
for all $\f \in C(X_{\tilde{A}},\Z)$.
\end{proof}

By combining 
Proposition \ref{prop:Phit} 
with
Lemma \ref{lem:Phi0A} and Lemma \ref{lem:Psietaxi},
we have the following proposition.
\begin{proposition} \label{prop:tildeA}
Let $\tilde{A}$ be the matrix defined by \eqref{eq:tildeA}
for a matrix $A$.
Then there exist an isomorphism
$\Phi:\SOTA 
\rightarrow 
\SOA$ 
satisfying  
$\Phi(\SDTA) = \SDA$,
a homomorphism 
$\Psi_\xi :C(X_{\tilde{A}},\Z) \rightarrow C(X_A,\Z)$
inducing an isomorphism
$\bar{\Psi}_\xi: H^{\tilde{A}} \rightarrow H^A$
of ordered abelian groups, 
and
a unitary one-cocycle
 $v_t^{\f} \in \U(M(\SOTA))$ relative to 
$\rho^{\tilde{A},\f} \otimes\id$
 for each  function
$\f\in C(X_{\tilde{A}},\Z)$
such that  
\begin{equation*}
\Phi \circ \Ad(v_t^{\f}) \circ (\rho^{\tilde{A},\f}_t\otimes \id)
 = (\rho^{A,\Psi_\xi(\f)}_t\otimes\id) \circ {\Phi}
\quad
\text{ for }
 t \in \RZ.
\end{equation*}
\end{proposition}
\begin{proof}
By Lemma \ref{lem:Phi0A}
there exists an isomorphism
\begin{equation*}
\Phi^0: P\OTA P \rightarrow \OA
\quad
\text{ such that }
\quad
\Phi^0(\DTA P) = \DA
\end{equation*}
and for $ \f \in C(X_{\tilde{A}},\Z)$
\begin{equation*}
\Phi^0 \circ \rho^{\tilde{A},\f}_t = \rho^{A, \Psi_\xi(\f)}_t\circ \Phi^0.
\end{equation*}
By Proposition \ref{prop:Phit}
there exist an isomorphism
\begin{equation*}
\tilde{\Phi}:\SOTA 
\rightarrow 
P{\mathcal{O}}_{\tilde{A}}P \otimes{\mathcal{K}}
\quad
\text{  such that }
\quad
 \tilde{\Phi}(\SDTA) = 
{\mathcal{D}}_{\tilde{A}}P \otimes{\mathcal{C}}
\end{equation*}
and 
a unitary one-cocycle
 $v_t^{\f} \in \U(M(\SOTA))$ 
 relative to 
$\rho^{\tilde{A},\f} \otimes\id$
such that 
\begin{equation*}
\tilde{\Phi} \circ \Ad(v_t^{\f}) \circ (\rho^{\tilde{A},\f}_t\otimes \id)
 = (\rho^{\tilde{A},\f}_t\otimes\id) \circ \tilde{\Phi}.
\end{equation*}
We set
\begin{equation*}
\Phi= ({\Phi^0}\otimes \id) \circ {\tilde{\Phi}} :
\SOTA 
\overset{\tilde{\Phi}}{\longrightarrow}
P{\mathcal{O}}_{\tilde{A}}P \otimes{\mathcal{K}}
\overset{{\Phi^0}\otimes \id}{\longrightarrow}
\SOA
\end{equation*}
so that 
$
\Phi(\SDTA)
= \SDA
$
and for $\f \in C(X_{\tilde{A}},\Z)$, we have
\begin{align*}
\Phi\circ \Ad(v^{\f}_t) \circ (\rho^{{\tilde{A}}, \f}_t \otimes\id)
=&  ({\Phi^0}\otimes \id) \circ {\tilde{\Phi}} 
\circ \Ad(v^{\f}_t) \circ (\rho^{{\tilde{A}}, \f}_t \otimes\id) \\
=& ({\Phi^0}\otimes \id) \circ (\rho^{{\tilde{A}}, \f}_t 
\otimes\id)\circ {\tilde{\Phi}}\\
=& [({\Phi^0}\circ \rho^{{\tilde{A}},\f}_t)\otimes \id ]
\circ {\tilde{\Phi}}\\
=& [(\rho^{A,\Psi_\xi(\f)}_t\circ 
{\Phi^0})\otimes \id ]\circ {\tilde{\Phi}}\\
=& (\rho^{A,\Psi_\xi(\f)}_t\otimes \id)\circ \Phi.
\end{align*}
\end{proof}

In \cite{PS}, Parry--Sullivan proved 
that the flow equivalence relation of topological Markov shifts 
is generated by the expansions $A \rightarrow \tilde{A}$
and topological conjugacies.
Therefore 
we arrive at the following theorem.
\begin{theorem}\label{thm:FETFAE}
If two-sided topological Markov shifts
$(\bar{X}_A, \bar{\sigma}_A)$ and $(\bar{X}_B,\bar{\sigma}_B)$
are flow equivalent,
then there exist an isomorphism
$\Phi:\SOA \rightarrow \SOB$,
a homomorphism
$\varphi:C(X_A,\Z) \rightarrow C(X_B,\Z)$
of ordered abelian groups
and 
a unitary one-cocycle
 $u_t^f \in \U(M(\SOA))$ relative to 
$\rho^{A,f} \otimes\id$
for each function
$f\in C(X_A,\Z)$
satisfying the following conditions
\begin{enumerate}
\renewcommand{\theenumi}{\roman{enumi}}
\renewcommand{\labelenumi}{\textup{(\theenumi)}}
\item
$\Phi(\SDA) = \SDB$,
\item
$\Phi \circ \Ad(u_t^f) \circ (\rho^{A,f}_t\otimes \id)
 = (\rho^{B,\varphi(f)}_t\otimes\id) \circ \Phi
\text{ for }
 f \in C(X_A,\Z), \, t \in \RZ,
$
\item
$\varphi:C(X_A,\Z) \rightarrow C(X_B,\Z)$
induces 
an isomorphism
from  
$H^A$
to  $H^B$
as ordered abelian groups.
\end{enumerate}
\end{theorem}
\begin{proof}
By the Parry--Sullivan's result,
we obtain the  assertion  from 
Proposition \ref{prop:tildeA} and
Corollary \ref{cor:SSE}.
\end{proof}
Recall that 
$\gamma^{A,f}_r \in \Aut(\OA), r \in \R$ 
stand for the one-parameter 
automorphisms for 
$f \in C(X_A,\R)$ 
defined by \eqref{eq:gammaArf}. 
We may similarly show the following result:
\begin{proposition}\label{prop:FETFAER}
If two-sided topological Markov shifts
$(\bar{X}_A, \bar{\sigma}_A)$ and $(\bar{X}_B,\bar{\sigma}_B)$
are flow equivalent, 
then there exists an isomorphism
$\Phi:\SOA \rightarrow \SOB$ such that 
$\Phi(\SDA) = \SDB$
and for each function
$f\in C(X_A,\R)$
there exists
a function
$\tilde{f}\in C(X_B,\R)$
and a unitary one-cocycle
 $u_r^f \in \U(M(\SOA))$ relative to 
$\gamma^{A,f} \otimes\id$
such that 
\begin{equation*}
\Phi \circ \Ad(u_r^f) \circ (\gamma^{A,f}_r\otimes \id)
 = (\gamma^{B,\tilde{f}}_r\otimes\id) \circ \Phi
\quad
\text{ for }
 r \in \R.
\end{equation*}
\end{proposition}

{\it Acknowledgments:}
The author would like to thank the referee for his careful reading the manuscript
and various suggestions. 
This work was supported by JSPS KAKENHI Grant Numbers 
23540237 and 15K04896.



\begin{thebibliography}{99}






















\bibitem{BH}
{\sc M. Boyle and D. Handelman},
{\it Orbit equivalence, flow equivalence and ordered cohomology},
Israel J.\  Math.\
{\bf 95}(1996), pp.\ 169--210.

\bibitem{BF} {\sc R. Bowen and J. Franks},
{\it Homology for zero-dimensional nonwandering sets},
Ann.\  Math.\ {\bf 106}(1977), pp.\ 73--92.





\bibitem{Brown}
{\sc L. G. Brown},
{\it Stable isomorphism of hereditary subalgebras of $C^*$-algebras},
Pacific J.\ Math.\ {\bf 71}(1977), pp.\ 335--348.



\bibitem{BGR}
{\sc L. G. Brown, P. Green and M. A. Rieffel},
{\it Stable isomorphism and strong Morita equivalence of $C^*$-algebras},
 Pacific  J.\ Math.\ {\bf 71}(1977), pp.\ 349--363.



\bibitem{Combes}
{\sc F. Combes},
{\it Crossed products and Morita equivalence},
Proc.\ London Math.\ Soc.\
{\bf 49}(1984), pp.\ 289--306.

\bibitem{CKRW}
{\sc  D. Crocker, A. Kumjian, I. Raeburn and D. P. Williams},
{\it An equivariant Brauer group and actions of groups of $C^*$-algebras},
J.\ Funct.\ Anal.\ {\bf 146}(1997), pp.\ 151--184.
























\bibitem{Cu2}{\sc J. Cuntz},
{\it Automorphisms of certain simple $C^*$-algebras},
 Quantum Fields-Algebras, Processes, Springer Verlag, Wien-New York, 1980,
pp.\ 187--196. 








\bibitem{CK}{\sc J. ~Cuntz and W. ~Krieger},
{\it A class of $C^*$-algebras and topological Markov chains},
 Invent.\ Math.\
 {\bf 56}(1980), pp.\ 251--268.





%

\bibitem{Franks}{\sc J. Franks},
{\it Flow equivalence of subshifts of finite type},
Ergodic Theory Dynam. Systems {\bf 4}(1984), pp.\ 53--66.



















\bibitem{Kitchens}{\sc B.~P. ~Kitchens},
{\it Symbolic dynamics}, 
Springer-Verlag, Berlin, Heidelberg and New York (1998).






\bibitem{LM}{\sc D. ~Lind and B. ~Marcus},
{\it An introduction to symbolic dynamics and coding},
 Cambridge University Press, Cambridge
(1995).









\bibitem{MaJOT2000} {\sc K. Matsumoto},
{\it On automorphisms of $C^*$-algebras associated with subshifts},
J. Operator Theory {\bf 44}(2000), pp.\  91--112.

\bibitem{MaETDS2004} {\sc K. Matsumoto},
{\it Strong shift equivalence of symbolic dynamical systems 
and Morita equivalence of $C^*$-algebras},
Ergodic Theory Dynam. Systems {\bf 24}(2002), pp.\ 199--215.


\bibitem{MaPacific}
{\sc K. Matsumoto},
{\it Orbit equivalence of topological Markov shifts and Cuntz--Krieger algebras},
Pacific J.\ Math.\ 
{\bf 246}(2010), 199--225.













\bibitem{MaPAMS}
{\sc K. Matsumoto},
{\it  Classification of Cuntz--Krieger algebras by orbit equivalence of topological Markov shifts},
Proc. Amer. Math. Soc.
{\bf 141}(2013), pp.\ 2329--2342.




\bibitem{MaPre2014}
{\sc K. Matsumoto},
{\it Strongly continuous orbit equivalence of 
one-sided topological Markov shifts},
J. Operator Theory 
{\bf 74}(2015), pp.\ 101--127.

\bibitem{MMKyoto}
{\sc K. Matsumoto and H. Matui},
{\it Continuous orbit equivalence of topological Markov shifts 
and Cuntz--Krieger algebras},
Kyoto J. Math.
{\bf 54}(2014), pp.\ 863--878.


\bibitem{MMPre2014}
{\sc K. Matsumoto and H. Matui},
{\it Continuous orbit equivalence of topological Markov shifts 
and dynamical zeta functions}, preprint, arXiv:1403.0719,
to appear in Ergodic Theory Dynam. Systems.




\bibitem{MatuiPLMS}
{\sc H. Matui}, 
{\it Homology and topological full groups of {\'e}tale groupoids on totally disconnected spaces},
 Proc. London Math. Soc. {\bf 104}(2012), 
 pp.\ 27--56.

\bibitem{MatuiPre2012}
{\sc H. Matui}, 
{\it Topological full groups of one-sided shifts of finite type},
J. Reine Angew. Math.
{\bf 705}(2015), pp. \ 35--84.





\bibitem{MPT}
{\sc P. S. Muhly, D.  Pask and M. Tomforde},
{\it Strong shift equivalence of $C^*$-correspondences},
Israel J. Math. {\bf 167} (2008), pp. \ 315--346. 


\bibitem{PS}
{\sc W. Parry and D. Sullivan},
{\it A topological invariant for flows on one-dimensional spaces},
Topology
 {\bf 14}(1975), pp.\ 297--299.









\bibitem{Po}
{\sc Y. T. Poon},
{\it A K-theoretic invariant for dynamical systems}, 
Trans.\  Amer.\ Math.\ Soc.\ 
{\bf 311}(1989), pp.\ 513--533.









\bibitem{RaeWill}  
{\sc I. Raeburn and D. P. Williams},
{\it Morita equivalence and continuous-trace $C^*$-algebras},
Mathematical Surveys and Monographs, vol(60)
Amer. Math. Soc, (1998).



\bibitem{Renault}{\sc J. Renault},
{\it A groupoid approach to $C^*$-algebras},
Lecture Notes in Math.  793, 
Springer-Verlag, Berlin, Heidelberg and New York (1980).


\bibitem{Rieffel1}
{\sc M. A. Rieffel},
{\it Induced representations of $C^*$-algebras},
Adv.\ in Math.\
{\bf 13}(1974), pp. 176--257.


\bibitem{Rieffel2}
{\sc M. A. Rieffel},
{\it  Morita equivalence for  $C^*$-algebras and  $W^*$-algebras},
J.\ Pure Appl.\ Algebra {\bf 5}(1974), pp.\ 51--96.  











\bibitem{Ro}
{\sc M. R{\o}rdam},
{\it Classification of Cuntz-Krieger algebras},
 K-theory {\bf 9}(1995), pp.\  31--58.






\bibitem{Tomforde}{\sc M. Tomforde}
{\it The Graph Algebra Problem Page : List of Open Problems},
http://www.math.uh.edu/tomforde/GraphAlgebraProblems/ListOfProblems.html.


















\bibitem{Williams} {\sc R. F. Williams},
{\it Classification of subshifts of finite type}, 
 Ann.\ Math.\  {\bf 98}(1973), pp.\ 120--153.
 erratum, Ann.\ Math.\
{\bf 99}(1974), pp.\ 380--381.

\end{thebibliography}
\end{document}